\documentclass[12pt]{article}

\usepackage{amsmath,amssymb,amsfonts,amsthm,mathrsfs,eucal,dsfont,verbatim}

\def\NN{{\Bbb N}}

\def\ZZ{{\Bbb Z}}

\def\CC{{\Bbb C}}

\newcommand{\E}{\mathds{E}}
\renewcommand{\P}{\mathds{P}}

\newcommand{\bX}{\mathbf{X}}

\newcommand{\bU}{\mathbf{U}}
\newcommand{\bV}{\mathbf{V}}
\newcommand{\bY}{\mathbf{Y}}
\newcommand{\bD}{\mathbf{D}}

\def\Ff{\mathcal{F}}

\renewcommand{\Re}{{\Bbb R}}
\def\eps{\varepsilon}

\def\1{1\!\!\hbox{{\rm I}}}

\def\ax{\Re^+}
\newcommand{\di}{\mathrm{d}}

\newcommand{\be}{\begin{equation}}
\newcommand{\ee}{\end{equation}}
\newcommand{\ba}{\begin{aligned}}
\newcommand{\ea}{\end{aligned}}

\newcommand {\ov}[1]{ \overline{#1}}
\newcommand {\wt}[1]{ \widetilde{#1}}
\newcommand {\wh}[1]{ \widehat{#1}}
\newcommand{\la}{\langle}
\newcommand{\ra}{\rangle}
\newcommand{\norm}[1] {|\!|\!|#1|\!|\!|}

\theoremstyle{plain}
\newtheorem{thm}{Theorem}[section]
\newtheorem{lem}{Lemma}[section]
\newtheorem{prop}{Proposition}[section]

\theoremstyle{definition}

\newtheorem{rem}{Remark}[section]

\numberwithin{equation}{section}

\begin{document}

\title{Well-Posedness,   Stability, and Sensitivities  for Stochastic Delay Equations: A Generalized Coupling Approach}%\footnote{\vica
%\textbf{thanks to grants!}\normal }}

\author{%
       \textsc{Alexei Kulik}
   \thanks{Institute of Mathematics, NAS of Ukraine, 3, Tereshchenkivska str., 01601  Kiev,
    \texttt{kulik.alex.m@gmail.com}}
         \quad \textsc{Michael Scheutzow}
\thanks{{Technische Universit\"at Berlin,
Institut f\"ur Mathematik, MA 7-5, Fakult\"at II,
Strasse des 17.~Juni 136, 10623 Berlin, \texttt{ms@math.tu-berlin.de}. Supported in part by DFG Research Unit FOR 2402  \bigskip
}}
   }

%\date{}

\maketitle
\begin{abstract}
We develop a new generalized coupling approach to the study of stochastic delay equations with  H\"older continuous coefficients,  for which analytical  PDE-based methods are not available. We prove that such equations possess unique weak solutions, and establish weak ergodic rates for the corresponding segment processes. We also prove, under additional smoothness assumptions on the coefficients, stabilization rates for the sensitivities in the initial value of the corresponding semigroups.
%Keywords: Markov process, generalized coupling, stochastic delay equation, sensitivity, stability\\
%Subject class: 60J05, 60J25,  37L40
\end{abstract}

\section{Introduction}
In this paper we introduce a new technique which makes it possible to study stochastic equations whose coefficients are assumed to be only \emph{H\"older continuous}, and which does not rely  on analytical results from the PDE theory. The analytic approach to the study of diffusion processes dates back to Kolmogorov, and nowadays is a common tool for the analysis of SDEs with low regularity of coefficients; e.g. \cite{Stroock_Varad}. For stochastic systems of more complicated structure, e.g. those described by  stochastic equations with delay, this approach is not realistic because of the necessity to study PDEs in (infinite-dimensional) functional spaces. For such systems, the It\^o-L\'evy stochastic approach is typically used which requires (one-sided local) Lipschitz continuity of the coefficients; e.g.~\cite{M84} or \cite{RS08}.  The current  paper shows that the range of application of the standard stochastic analysis tools can be substantially extended, including {delay} equations with low regularity of coefficients.

Our approach is based on the concept of \emph{generalized coupling}, which extends the classical notion of \emph{coupling} in the following way. By definition, a coupling is a probability measure on a product space with prescribed marginal distributions. For a \emph{generalized} coupling the marginals satisfy instead milder deviation bounds from the prescribed distributions. The class of generalized couplings is much wider than of classical couplings, and it is typically much easier to construct  for a given system a generalized coupling with desired properties than a true one, for more details see Section \ref{s3} below. This makes generalized couplings quite an efficient tool in the ergodic theory of Markov processes, see the recent paper \cite{BKS18}  where they were used as a key ingredient in the construction of contracting/nonexpanding  distance-like functions for complicated SPDE models.

In \cite{BKS18}, generalized couplings were first constructed using stochastic control arguments, and then used for the construction of true couplings; in this last step the change of
the marginal laws caused by the control terms was in a sense reimbursed. We call this type of argument a \emph{Control-and-Reimburse (C-n-R)} strategy. The same general idea ---
to apply a stochastic control in order to improve the system, and then to take into account the impact of the control --- is scattered in the literature; e.g. it is used in  \cite[Section~5.2]{HMS11} in a construction of of contracting/nonexpaning distance-like function $d(x,y)$ for delay equations, in \cite{H02} in an approach to the study of weak ergodicity of SPDEs, in \cite{AVer} in the proof of  ergodicity in total variation  for degenerate diffusions, and in  \cite{BodKul08} in the proof of  ergodicity in total variation for solutions to L\'evy driven SDEs. Related ideas were used to establish the Harnack inequality for SDEs and SFDEs \cite{Wang, ERS09}.

We further develop this general idea in the following two directions. First, we show that the Control-and-Reimburse (C-n-R) strategy  is well applicable under just  H\"older continuity assumptions on the coefficients (actually, one-sided H\"older continuity for the drift). This makes it possible to establish ergodic rates for delay equations with non-Lipschitz coefficients; moreover, essentially the same  generalized coupling construction allows one to
prove well-posedness of the system, i.e.~that the weak solution to the equation is uniquely defined and the corresponding \emph{segment process}  is a time-homogeneous Markov  process with the Feller property. Second, we establish  stabilization rates for \emph{sensitivities} for the model; that is, for the derivatives of the semigroup rather than for the semigroup itself. The natural and commonly adopted way to get such rates in a finite-dimensional setting is based on the Bismut-Elworthy-Li type formulae (\cite{Bismut}, \cite{ElLi}) which give integral representations of sensitivities based on the integration-by-parts formulae. Such a \emph{regularization} effect in an infinite-dimensional setting becomes much more structure demanding, since the random noise (which is the source of the integration-by-parts formula) needs to be non-degenerate in the entire space; for one result of such type and a detailed discussion we refer to  \cite{Cer99}, where reaction-diffusion equations with a cylindrical noise are considered. In the delay case the noise is finite-dimensional and thus is strongly degenerate; hence the Bismut-Elworthy-Li type formula for the (Fr\'echet) derivatives of the semigroup is hardly available. Nevertheless, employing the C-n-R strategy we are able to derive a family of representation formulae for these derivatives, which can be understood as ``poor man's Bismut-Elworthy-Li type formulae'', see \eqref{der_rep} and \eqref{int_rep_der_k}. Namely, these formulae are not completely free from  gradient terms like $\nabla f$, but the weights in the corresponding integral expressions can be forced to decay exponentially fast at an arbitrarily large rate. Using these representation formulae we manage to establish  stabilization rates for sensitivities (derivatives) of arbitrary order; note that the (full) regularization effect now has no reason to appear, and thus for these results we have to assume certain smoothness of the coefficients.

The structure of the paper is the following. In Section \ref{s2} the main results are formulated and briefly discussed. To make the exposition transparent, we explain in a separate Section \ref{s3} the cornerstones of the proofs. The detailed proofs of the three main groups of results are given in Section \ref{s4}, Section \ref{s5}, and Section \ref{s6}, respectively.

\noindent \textbf{Acknowledgments}. The authors are grateful to Ren\'e Schilling for helpful discussions; the first version of Lemma \ref{lem1} was actually suggested by him.  The work on the project has been finished  during the visit of AK to the Technical University of Dresden (Germany); AK is very grateful  to the Technical University of Dresden and especially to Ren\'e Schilling for their support and hospitality.

\section{Main results}\label{s2}
\subsection{Weak solution: existence and uniqueness}

%\subsection{The model}
Let $n\in\NN$ and $r>0$. Denote by $\CC=C([-r,0],\Re^n)$ the space of continuous functions with the supremum norm $\|\cdot\|$. For a stochastic process $X=\{X(t), t\geq -r\}$ in $\Re^n$ define the corresponding \emph{segment process} $\bX=\{\bX_t, t\geq 0\}$ in $\CC$ by
$$
\bX_t=\{X(t+s), s\in[-r,0]\}\in \CC, \quad t\geq 0.
$$
Consider the stochastic delay differential equation (SDDE)
\be\label{SDDE}
\di X(t)=a(\bX_t)\, \di t +\sigma(\bX_t)\, \di W(t),\quad t\ge0,
\ee
with the  initial condition $\bX_0=x\in \CC$. Here $W$ is a Brownian motion in $\Re^m$, $m\geq 1$, and $a:\CC\to\Re^n$ and $\sigma:\CC\to\Re^{n\times m}$ are given functions. We will focus on \emph{weak} solutions; that is, processes $X$ with continuous trajectories such that \eqref{SDDE} holds true with  \emph{some} Wiener process $W$.
%The main result of this section provides  conditions on the coefficients $a$ and $\sigma$, sufficient for the process $X$ to be well defined as the %{weak} solution to \eqref{SDDE}.

Our main assumptions are listed below.

$\mathbf{H}_1.$ The function $a$ is continuous, bounded on bounded subsets of $\CC$, and satisfies the following \emph{finite range one-sided H\"older conditio}n with index $\alpha>0$: there exists $C$ such that
\be\label{H_gamma}
\Big(a(x)-a(y), x(0)-y(0)\Big)\leq C\|x-y\|^{\alpha+1}, \quad \|x-y\|\leq 1.
\ee
Here and below we denote the scalar product in $\Re^n$ by $(\cdot, \cdot)$.

$\mathbf{H}_2.$ The function $\sigma$ satisfies the following \emph{finite range H\"older condition} with index $\beta>1/2$:
\be\label{H_delta}
\norm{\sigma(x)-\sigma(y)}\le  C\|x-y\|^\beta, \quad \|x-y\|\leq 1.
\ee
Here and below  $\norm{\cdot}$ denotes the Frobenius norm of a matrix, $\norm{M}:=\sqrt{\sum M_{ij}^2}$.

$\mathbf{H}_3.$ For each $x\in \CC$ there exists a right inverse $\sigma(x)^{-1}$ of the matrix $\sigma(x)$, and
\begin{equation}\label{nd}
\sup_{x\in \CC} \norm{\sigma(x)^{-1}}<\infty.
\end{equation}

$\mathbf{H}_4.$ The following  one-sided linear growth bound for $a$ holds:
\be\label{gba}
\Big(a(x), x(0)\Big)\leq C(1+\|x\|^{2}), \quad x\in \CC.
\ee

Note that a similar linear growth bound for $\sigma$ holds true by \eqref{H_delta}:
\be\label{gbsigma}
\norm{\sigma(x)}\leq C(1+\|x\|), \quad x\in \CC.
\ee

\begin{thm}\label{t1} Assume $\mathbf{H}_1$ -- $\mathbf{H}_4$. Then the following statements hold.

 \begin{enumerate}
   \item[1.] For any $x\in \CC$ there exists a weak solution $X$ to \eqref{SDDE} with $\bX_0=x$.

   \item[2.] The weak solution to \eqref{SDDE} is unique in law; that is, any two such solutions with the same initial segment $x\in \CC$ have the same law in $C([-r, \infty), \Re^n)$.

   \item[3.] The segment process $\bX$, which corresponds to the weak solution to \eqref{SDDE},   is a time-homogeneous Markov process in $\CC$, which has the Feller property.
 \end{enumerate}
  \end{thm}

 The main  difficulty in this theorem is the uniqueness statement 2. We note that by a slight modification of the proof one can get the same result
 assuming $a$ being just continuous and bounded on bounded subsets (that is, allowing $\alpha=0$ in $\mathbf{H}_1$). This minor improvement however does not apply to Theorem \ref{t2} below, and in order to keep the exposition reasonably short we thoroughly explain the one generalized coupling construction which suites well for both these results, and requires $\alpha>0$.

\subsection{Ergodic rates for the segment process} Let $d(\cdot, \cdot)$ be a metric on $\CC$. The corresponding \emph{coupling} (or \emph{minimal}) distance on the set $\mathcal{P}(\CC)$ of probability distributions on $\CC$ is given by
\begin{equation}\label{coupling_distance}
d(\mu,\nu)=\inf_{\lambda\in\mathcal{C}(\mu,\nu)}\int_{\CC\times \CC} d(x,y) \lambda(\di x,\di y),\quad \mu,\nu\in\mathcal{P}(\CC).
\end{equation}
Here $\mathcal{C}(\mu,\nu)$ denotes the set of all \textit{couplings} between $\mu$ and $\nu$, i.e. probability measures on $\CC\times \CC$ with marginals $\mu$ and $\nu$. In what follows, we will consider $d(\cdot, \cdot)$ on $\CC$ which is equivalent to the usual distance $\|\cdot-\cdot\|$ and is bounded. In this case the corresponding coupling distance is a metric, and convergence in this metric is equivalent to weak convergence in $\mathcal{P}(\CC)$. The famous Kantorovich-Rubinshtein theorem provides an alternative expression for $d(\mu, \nu)$: denote for $f:\CC\to \Re$
\be\label{Lip_comt}
\|f\|_{\mathrm{Lip}_d}=\sup_{x\not=y}{|f(x)-f(y)|\over d(x,y)},
\ee
then
\be\label{KR_thm}
d(\mu, \nu)=\sup_{f:\|f\|_{\mathrm{Lip}_d}=1}\left|\int f\di \mu-\int f \di \nu\right|.
\ee
In the literature, $d(\mu,\nu)$ is frequently called the \textit{1-Wasserstein} distance, though the name \textit{Kantorovich} distance is historically more appropriate.

In this section we will establish weak ergodic rates  for the segment process $\bX_t, t\geq 0$ with respect to a properly chosen coupling distance $d(\cdot, \cdot)$. That is, we will give sufficient conditions for $X$ to have a unique invariant probability measure (IPM) $\pi$ and quantitative bounds for the convergence
$$
d(P_x^t, \pi)\to 0, \quad t\to 0;
$$
here and below we denote by
$$
P_x^t(A)=\P_x(\bX_t\in A), \quad A\in \mathcal{B}(\CC), \quad x\in \CC, \quad t\geq 0
$$
the transition probability for the segment process.  We adopt the method introduced in \cite{HMS11} and further developed in \cite{Bu14}, \cite{DFM}, \cite[Chapter 4]{K17}. The method is  based on a proper combination of contraction, non-expansion, and recurrence properties, which we briefly explain here.  Fix a time discretization step $h>0$ and consider the \emph{skeleton chain} $\bX^h=\{\bX_{kh}, k\in \ZZ_+\}$ for the segment process $\bX$. The distance $d(x,y)$ is called \emph{contracting for $\bX^h$ on a set $B\subset \CC\times \CC$}, if there exists $\theta\in (0,1)$ such that
\be\label{contracting}
d(P_x^h, P_y^h)\leq \theta d(x,y), \quad (x,y)\in B.
\ee
The distance $d(x,y)$ is called \emph{ non-expanding for $\bX^h$}, if
\be\label{non-expanding}
d(P_x^h, P_y^h)\leq d(x,y), \quad x,y\in \CC.
\ee
With a slight abuse of terminology, we will say that a set $K\subset\CC$ is \emph{$d$-small for $\bX^h$} if $d$ is non-expanding for $\bX^h$ and is contracting on $K\times K$ (this definition differs from the original one {\cite[Definition~4.4]{HMS11}}, but has essentially the same scope and is technically more convenient).

The crucial question in the entire approach is how to construct  a non-expanding metric $d$, which in addition is contracting on a sufficiently large class of sets.
The following theorem, which is the main result of this section, resolves this question for the SDDE \eqref{SDDE}.
Denote for $x,y\in \CC$
$$
d_{N,\gamma}(x,y)=\Big(N\|x-y\|^{\gamma}\Big)\wedge 1,  \quad N\geq 1, \quad \gamma\in (0,1].
$$
Clearly, each $d_{N, \gamma}$ is a  metric on $\CC$.

\begin{thm}\label{t2} I. Assume $\mathbf{H}_1$ -- $\mathbf{H}_4$.  Then for any $h>r$ and
$
\gamma \in \big(0,\alpha\wedge (2\beta-1)\big)
$
there exists $N_{h,\gamma}$ such that for any $N\geq N_{h,\gamma}$ any bounded set $K\subset \CC$ is $d_{N,\gamma}$-small for $\bX^h$.

II. Assume in addition that the following stronger version of $\mathbf{H}_4$ holds true:
\be\label{gba1}
\Big(a(x), x(0)\Big)\leq C(1+|x(0)|^{2}), \quad x\in \CC,
\ee
\be\label{gbsigma1}
\norm{\sigma(x)}\leq C(1+|x(0)|), \quad x\in \CC.
\ee
Then for any $h>r$ and  positive
$
\gamma<\alpha\wedge (2\beta-1)
$
there exists $N_{h,\gamma}$ such that for any $N\geq N_{h,\gamma}$ each set
$$
H_c=\{x\in \CC:|x(0)|\leq c\}, \quad c\geq 0
$$
is $d_{N,\gamma}$-small for $\bX^h$.
\end{thm}

  Once a proper non-expanding metric $d$ is constructed, the general theory can be applied which allows one to obtain (weak) ergodic rates,  taking into account recurrence properties of the process and measuring how quickly the system visits a $d$-small set; e.g. \cite[Section 4.5]{K17}. Namely, we have the following statement. Denote
\be\label{dg}
  d_\gamma(x,y)= d_{1,\gamma}(x,y)=\|x-y\|^\gamma\wedge 1,
  \ee
  and for a given measurable function $\phi:\Re^+\to \Re^+$ define the functions
  $$
  \Phi(v)=\int_1^v{dw\over \phi(w)}, \quad r(t)=\phi(\Phi^{-1}(t)).
  $$

  \begin{thm}\label{t23}  Assume $\mathbf{H}_1$ -- $\mathbf{H}_4$. Assume also that, for some $h>r$, the following \emph{Lyapunov-type condition} holds:
  \be\label{lyap}
\E_x V(\bX_h)-V(x)\leq -\phi\Big(V(x)\Big)+C_V, \quad x\in \CC.
\ee
Here  $V:\CC\to [1, +\infty)$ is a measurable \emph{Lyapunov function}, $C_V$ is a constant, and the  function $\phi:\Re^+\to \Re^+$ with $\phi(\infty)=\infty$ is concave and strictly increasing. Assume that either
  \item \be\label{level_sets} V(x)\to \infty,\quad \|x\|\to \infty,
\ee
or
 \be\label{level_sets1} V(x)\to \infty,\quad |x(0)|\to \infty
\ee
and in addition \eqref{gba1}, \eqref{gbsigma1} hold true.

Then there exists a unique IPM $\pi$ for the segment process $\bX$, and for any $
 \gamma\in (0,1],  \delta\in(0,1)$
there exist $c, C>0$ such that
%\be\label{rate_pair}
%d_\gamma(P_x^t, P_y^t)\leq {C\over \Psi(c t)^{\delta}}\Big(V(x)^{\delta}+V(y)^{\delta}\Big)d_\gamma(x,y),\quad x,y\in \CC,
%\ee
\be\label{rate_ind}
d_\gamma(P_x^t, \pi)\leq {C \over r(c t)^\delta}\phi(V(x))^{\delta}, \quad x\in \CC.
\ee

  \end{thm}

  \begin{rem} The bound \eqref{rate_ind} can be alternatively considered as a convergence rate for the semigroup
  $$
  P_tf(x)=\E_x f(\bX_t), \quad x\in \CC, \quad t\geq 0.
  $$
  Namely, denote for $\gamma\in (0,1]$
$$
\|f\|_{H_\gamma}=\sup_{0<\|x-y\|\leq 1}{|f(x)-f(y)|\over \|x-y\|^\gamma}+\sup_{\|x-y\|> 1}|f(x)-f(y)|,
$$
which is   just the Lipschitz constant of $f$ w.r.t. $d_\gamma$; see \eqref{Lip_comt} and \eqref{dg}. Let also $H_\gamma$ denote the class of functions $f:\CC\to \Re$ with $\|f\|_{H_\gamma}<\infty$. Then by
\eqref{KR_thm} inequality \eqref{rate_ind} is equivalent to the following:
\be\label{rate_Hol}
P_tf(x)\leq {C \over r(c t)^\delta}\phi(V(x))^{\delta}\|f\|_{H_\gamma}, \quad x\in \CC, \quad t\geq 0, \quad f\in H_\gamma.
\ee
  \end{rem}

In general, it is a separate non-trivial question how to verify the Lyapunov condition \eqref{lyap} for delay equations. We do not address this question here, referring to \cite{BS} and references therein. Note however, that there are simple models, where this condition can be checked essentially in the same way as in the (non-delayed) diffusion setting. In the following proposition, $\alpha,b,c$ denote some positive constants, whose values can be  specified in terms of the other parameters.

\begin{prop} Let the coefficient $\sigma(\cdot)$ be bounded, and the coefficient  $a(\cdot)$ satisfy
$$
\Big(a(x), x(0)\Big)\leq -A_\kappa |x(0)|^{\kappa+1}, \quad |x(0)|\geq R
$$
for some $\kappa\geq -1$, $A_\kappa<0$, and $R>0$. Assume also $\mathbf{H}_1$ -- $\mathbf{H}_3$. Then
\begin{itemize}
  \item[(i)] If $\kappa\geq 0$, \eqref{lyap} holds true with
  $$
  V(x)=e^{\alpha|x(0)|}, \quad \phi(v)=cv.
  $$
     \item[(ii)] If $\kappa\in (-1, 0)$, \eqref{lyap} holds true with
  $$
  V(x)=e^{\alpha|x(0)|^{\kappa+1}}, \quad \phi(v)=cv\Big(\log v+b\Big)^{2\kappa/(\kappa+1)}.
  $$
  \item[(iii)] If $\kappa=-1$ and in addition $2A_{-1}>\Lambda:=\sup_x\norm{\sigma(x)}^2$, then for
  $p>2$,
  $$
  p<2+(2A_{-1}-\Lambda)(\sup_x\|\sigma(x)\|^2)^{-1}$$ condition \eqref{lyap} holds true with
  $$
  V(x)=|x(0)|^{p}, \quad \phi(v)=av^{1-2/p}.
  $$
\end{itemize}
\end{prop}
The proof is analogous to the one of \cite[Theorem 3.3]{Bu14}; see also \cite [Proposition 4.6.1]{K17}.

Finally, let us mention that, without assuming a Lyapunov function to exist, we still have the following stabilization property: if
$\mathbf{H}_1$ -- $\mathbf{H}_4$ hold and there exists \emph{some} IPM $\pi$ for the segment process $\bX$, then this IPM is \emph{unique} and
$P_x^t\to \pi$ weakly as $t\to \infty$ for every $x$. This follows from \cite[Theorem 2.4]{HMS11} and a slightly re-arranged argument from the proof of  Theorem \ref{t2}; see Remark \ref{r52} below. Note that the support statement \eqref{support} proved below provides the property $\mathrm{supp}\,\pi=\CC$ required for \cite[Theorem 2.4]{HMS11} to apply. Alternatively, one can refer to the approach developed in \cite{KPS10} and \cite{BKS}, based on the notion of \emph{e-processes}; it is easy to see that  Remark \ref{r52} yields the e-process property for $\bX_t, t\geq 0$.

\subsection{Sensitivities w.r.t. the initial condition: integral representation and stabilization}
Denote by $C^k(\CC)$ the class of $k$ times Fr\'echet differentiable functions $f:\CC\to \Re$ with continuous  derivatives.
The $k$-th  derivative and direction-wise derivatives for $f\in C^k(\CC)$ will be denoted by $\nabla^k f$ and
$$
\nabla_{z_1, \dots, z_k}^k f=\la\nabla\dots\la \nabla f, z_1\ra, \dots z_k\ra, \quad z_1,\dots, z_k\in \CC,
$$
respectively. Similarly the classes $C^k(\CC, \Re^n)$ and $C^k(\CC, \Re^{n\times m})$ of the functions valued in  $\Re^n$ and $\Re^{n\times m}$ are defined, and the notation for the derivatives is the same. By $C^k_b(\CC),$ we denote the class of  $C^k(\CC)$ functions, bounded with their derivatives up to order $k$.

For a fixed $k\geq 1$, assume the following.

 $\mathbf{C}^{(k)}$. $a\in C^k(\CC, \Re^n)$, $\sigma \in C^k(\CC, \Re^{n\times m})$, and their derivatives of the  orders $1, \dots, k$ are bounded and uniformly continuous on $\CC$.

We first consider the case $k=1$.  Define for $\lambda\geq 0, z\in \CC$ the process $U^{\lambda,z}$ as the solution to the SDDE
 \be\label{U}
\di U^{\lambda,z}(t)=\la \nabla a(\bX_t), \bU^{\lambda,z}_t\ra \, \di t +\la \nabla \sigma(\bX_t), \bU^{\lambda,z}_t\ra\, \di W(t)-\lambda U^{\lambda,z}(t)\,\di t,\quad t\ge0
\ee
with the initial condition $\bU^{\lambda,z}_0=z.$

\begin{thm}\label{t4} Let $\mathbf{C}^{(1)}$ and $\mathbf{H}_3$ hold true. Then for any $f\in C_b^1(\CC)$ the functions
$P_tf,  t\geq 0
$
belong to $C^1_b(\CC)$. For any $\lambda\geq 0, z\in \CC$ the following representation formula holds:
\be\label{der_rep}
\nabla_z \E_x f(\bX_t)=\E_x \Big\la \nabla f(\bX_t), \bU^{\lambda,z}_t\Big\ra+\lambda \E_x \left(f(\bX_t)\int_0^t\sigma(\bX_s)^{-1}U^{\lambda,z}(s)\,dW(s)\right).
\ee
\end{thm}

Combining the representation formula \eqref{der_rep} and  Theorem \ref{t23}, we get the following stabilization bound for $\nabla P_tf $ as $t\to \infty$.

\begin{thm}\label{t5} Let $\mathbf{C}^{(1)}$ and the assumptions of Theorem \ref{t23} hold true. Then for  any $\gamma\in (0,1]$, $\delta\in(0,1)$, and $Q>0$
there exist  constants $c_Q>0, C>0$ such that for any $f\in C^1_b(\CC)$
%\be\label{rate_pair}
%d_\gamma(P_x^t, P_y^t)\leq {C\over \Psi(c t)^{\delta}}\Big(V(x)^{\delta}+V(y)^{\delta}\Big)d_\gamma(x,y),\quad x,y\in \CC,
%\ee
\be\label{der_rate_ind}
\|\nabla P_t(x)\|\leq C \left( {\log r(c_Q t)+\phi(V(x))\over r(c_Q t)}\right)^\delta \|f\|_{H_\gamma}+
Ce^{-Q t}\sup_{y\in \CC}\|\nabla f(y)\|, \ \quad t\geq 0.
\ee
\end{thm}

\begin{rem}\label{r22}
Note that the first term in the right hand side of \eqref{der_rate_ind} coincides with the bound \eqref{rate_Hol} up to minor changes (extra logarithmic term and a time change $ct\mapsto c_Qt$), which do not affect the structure of the estimate. The derivative  $\nabla f$ is involved in the second term only, and this term is decaying very rapidly: at exponential rate, and the index $Q$ in this rate can be made arbitrarily large.
\end{rem}

Next, let $k>1$ be arbitrary. For $f\in C^k(\CC)$ and $j=1, \dots. k$, for any $x\in C$ one can naturally treat $\nabla^j f(x)$ as a $j$-linear form on $\CC$. We endow the space of such forms by the usual norm
$$
\|L\|_j=\sup_{\|z_1\|=\dots=\|z_j\|=1}|L(z_1, \dots, z_j)|,
$$
and denote for $f\in C^k_b(\CC)$
$$
\|f\|_{(k)}=\sup_{x\in \CC}\sum_{j=1}^k\|\nabla^j f(x)\|_j;
$$
note that $\|\cdot\|_{(k)}$ is actually a seminorm because the values of $f$ itself are not involved in it.

\begin{thm}\label{t6} Let  $\mathbf{C}^{(k)}$ for some $k>1$ and let the assumptions of Theorem \ref{t23} hold true. Then for  any $\gamma\in (0,1]$, $\delta\in(0,1)$, and $Q>0$
there exist  constants $c_k>0, C>0$ such that for any $f\in C^k_b(\CC)$
%\be\label{rate_pair}
%d_\gamma(P_x^t, P_y^t)\leq {C\over \Psi(c t)^{\delta}}\Big(V(x)^{\delta}+V(y)^{\delta}\Big)d_\gamma(x,y),\quad x,y\in \CC,
%\ee
\be\label{der_rate_ind_k}
\|\nabla^k P_t(x)\|_k\leq C \left( {\log r(c_k t)+\phi(V(x))\over r(c_k t)}\right)^\delta \|f\|_{H_\gamma}+
Ce^{-Q t}\|f\|_{(k)},  \quad t\geq 0.
\ee
\end{thm}

Note that the structure of the estimate for the higher order derivatives remains exactly the same as for the 1st order one: the first term essentially coincides with \eqref{rate_Hol} and contains the $H_\gamma$-seminorm of $f$, only, while the second term, which contains the $\|f\|_{(k)}$-seminorm, decays exponentially fast. We mention that there exists an  integral representation for the higher order derivatives, analogous
to \eqref{der_rep}, see \eqref{int_rep_der_k}; actually, the proof of Theorem \ref{t6} is based on this representation. However, this representation is now less explicit and more cumbersome, that is why we do not formulate it separately here.

\section{Outline: generalized couplings and the Control-and-Re\-im\-bur\-se strategy}\label{s3} Within the classical \emph{coupling} approach to the study of ergodic properties of Markov systems one has to construct, on a common probability space, a pair of stochastic processes with prescribed law, such that the distance between the components of the pair obeys
certain bounds. For instance,  inequality \eqref{contracting} implies that for any $(x, y)\in B$ there exists a pair of segment processes $\bX, \bY$ with $\mathrm{Law}\, (\bX)=P_x, \mathrm{Law}\, (\bY)=P_y$ such that
\be\label{contracting2}
\E d(\bX_h, \bY_h)\leq \theta d(x,y).
\ee
The key question is how to construct a   pair $(\bX, \bY)$ with the required properties.  One naive way  is to take the coupling which consists of two solutions to equation \eqref{SDDE} with the same noise $W$ and given initial conditions $x,y$. This natural coupling is often not a good choice.
% appears to have quite limited field of of applications.
Namely, assume for the moment that the coefficients $a$, $\sigma$ are Lipschitz continuous, then such a pair is well defined, but the \emph{contraction} property \eqref{contracting2} in general has no reason to hold true. Namely, let us consider, for illustration purposes only, the much simpler model of a non-delayed scalar equation:
\be\label{SDE}
\di X(t)=a(X(t))\, \di t +\sigma(X(t))\, \di W(t),\quad t\ge0.
\ee
Then a simple calculation using   It\^o's formula shows that
$$
\E (X(t)-Y(t))^2\leq e^{\kappa t}(X(0)-Y(0))^2, \quad \kappa=\sup_{u\not=v}{2(a(u)-a(v))(u-v)+(\sigma(u)-\sigma(v))^2\over |u-v|^2}
$$
Hence for \eqref{contracting2} to hold true with (say) $d(x,y)=|x-y|$ one needs $\kappa<0$; that is, the system has to be \emph{dissipative}, which is a strong structural limitation. In the non-delayed case the \emph{synchronous } (or \emph{marching}) coupling  construction  outlined above is far from being optimal: there are much better possibilities, which exploit the analytical properties of the corresponding semigroup, provided by parabolic PDE theory. In the delay case such an analytical theory is not available, while the {synchronous coupling} admits a useful modification, which we further explain in detail.

In what follows, let $x,y\in \CC$ be given, and $X$ be the solution to \eqref{SDDE} with $\bX_0=x$. Next, let $Y$ be a solution to the following \emph{controlled} version of \eqref{SDDE}: $\bY_0=y$,
\be\label{SDDE-Y}
\di Y(t)=a(\bY_t)\, \di t +\sigma(\bY_t)\, \di W(t)+\chi(t)\,\di t ,\quad t\ge0,
\ee
where the \emph{control term} $\chi$ is yet to be chosen. For instance, still assuming $a, \sigma$ to be Lipschitz continuous, we can take
\be\label{control_term}
\chi(t)=\lambda\big(X(t)-Y(t)\big).
\ee
Then, for $\lambda$ large enough in comparison to the Lipschitz constants for $a, \sigma$,  there exist $C(\lambda), \kappa(\lambda)>0$ such that
\be\label{HMS}
\E\|\bX_t-\bY_t\|^2\leq C(\lambda)e^{-\kappa(\lambda) t}\|x-y\|^2, \quad t\geq 0,
\ee
see \cite[Lemma~3.6]{HMS11}. This  actually gives the contraction property \eqref{contracting2} for $d(x,y)=\|x-y\|$ and $h$ large enough. In a sense,  by endowing the original equation with the control term \eqref{control_term},  we  improve the original system from non-dissipative to a  dissipative one.

Clearly, the pair $\bX, \bY$ is not a (true) coupling: since equation \eqref{SDDE-Y} contains an extra term  $\chi_t\,\di t$, the law of $\bY$ has no reason to coincide with $P_y$. However, there is still some link between these laws, which is the reason for us to call the pair $\bX, \bY$ a \emph{generalized coupling}. Namely, assume that the non-degeneracy assumption $\mathbf{H}_3$ holds true, then \eqref{SDDE-Y} can be written in the form
\be\label{SDDE-Y-tilde}
\di Y(t)=a(\bY_t)\, \di t +\sigma(\bY_t)\, \di \wt W(t),\quad t\ge0,\quad \bY_0=y
\ee
with
$$
\di \wt W(t)=\di W(t)+\eta(t)\di t, \quad \eta(t)=\sigma(\bY_t)^{-1}\chi(t).
$$
Note that, by \eqref{HMS} and $\mathbf{H}_3$,
$$
\E\int_0^\infty |\eta(t)|^2\, \di t\leq C\|x-y\|^2.
$$
Then  the law of $\wt W$ on $C([0, \infty), \Re^m)$ is absolutely continuous w.r.t. the law of $W$ and moreover the following bound for the total variation distance holds, see Theorem \ref{tKLGirsanov} and \eqref{otzenka2}:
$$
d_{TV}\Big(\mathrm{Law}( W), \mathrm{Law}(\wt W)\Big)\leq {1\over 2}\sqrt{\E\int_0^\infty |\eta(t)|^2\, \di t}.
$$
Since $\bY$ is the strong solution to \eqref{SDDE-Y-tilde}, it can be understood as an image of $\wt W$ under a measurable mapping $\Phi$ such that the image of $W$ under $\Phi$  is just the solution to
$$
\di Y(t)=a(\bY_t)\, \di t +\sigma(\bY_t)\, \di  W(t),\quad t\ge0,\quad \bY_0=y.
$$
Hence
\be\label{TV}
d_{TV}\Big(\mathrm{Law}(\bY), P_y\Big)\leq d_{TV}\Big(\mathrm{Law}( W), \mathrm{Law}(\wt W)\Big)\leq C\|x-y\|,
\ee
which in particular means that the  change of the law of the solution, caused by the additional stochastic control term \eqref{control_term}, becomes smaller when $x,y$ become closer.  This observation enables us to construct a new (true) coupling from the generalized one $\bX,\bY$, which satisfies \eqref{contracting2} with properly chosen $d$ and $D$; see \cite[Theorem 2.4]{BKS18} and Proposition \ref{p_coup} below.

Let us summarize: because the direct construction of a coupling with the required properties may be difficult, we first construct a generalized one. At this stage, using additional control-type  terms, the properties of the system can be improved, e.g.  a contraction-type bound \eqref{HMS} can be provided for a non-dissipative system. Then we in a sense
reimburse the changes to the marginal laws,  generated by the control-type terms, using e.g.~the bound   \eqref{TV} and constructing a true coupling from the generalized one. This is the essence of the two-stage C-n-R strategy mentioned in the introduction.

The C-n-R strategy appears to be quite flexible; now we explain how it can be applied to the study of sensitivities. Under
the condition $\mathbf{C}^{(1)}$ the solution to \eqref{SDDE} is \emph{$L_p$-Fr\'echet differentiable} w.r.t. $x\in \CC$; see Section \ref{s61} for the corresponding definition and proofs. The respective derivative in the direction $z\in \CC$ equals just $U^{0, z}$, which clearly yields \eqref{der_rep} with $\lambda=0$. However, in order for the latter identity to provide the stabilization of the sensitivity as $t\to \infty$, it is required that $\bU_t^{0,z}\to 0, t\to \infty$. This can be guaranteed under an additional \emph{confluence} assumption, which is an analogue of the dissipativity assumption for the gradient process, see \cite{PP14}
 for a systematic treatment of  confluent SDEs. Using generalized couplings, we  avoid using this strong  additional assuption. Namely, together with the \emph{true} derivative $\bU^{0, z}_t$ in the direction $z$, we construct a family of \emph{controlled} derivatives $\bU^{\lambda, z}_t$, which are the limits of
 $$
 {\bY^{\lambda, x+\eps z}_t-\bX^{x}_t\over \eps}
$$
where $\bY^{\lambda, x+\eps z}$ is defined by \eqref{SDDE-Y} with (slightly changed) control term \eqref{control_term} and the initial value $x+\eps z$.  We have for $\lambda>0$ large enough $\bU_t^{\lambda,z}\to 0, t\to \infty$ exponentially fast; that is, using the control-type argument we actually transform a non-confluent system to a (sort of) confluent one. The ``reimbursement'' for such a control is represented by the  additional term in \eqref{der_rep}, which appears due to the Girsanov formula.

In all the previous considerations, it was assumed that the coefficients $a, \sigma$ are Lipschitz continuous. This limitation is not ultimate, and in some cases  can be  substantially relaxed; we postpone the detailed discussion of this subtle point to Section \ref{s51} below.  Here we just outline the main idea we use in the  the proof of  weak uniqueness for \eqref{SDDE}. Let us take
\emph{some} weak solution to \eqref{SDDE}, and consider a family of approximate SDDEs of the form
\be\label{SDDE-Y-eps}
\di Y^\eps(t)=a^\eps(\bY_t^\eps)\, \di t +\sigma^\eps(\bY_t^\eps)\, \di W(t)+\chi^\eps(t)\di t ,\quad t\ge0
\ee
with the same $W$,   the same initial value, and Lipschitz continuous $a^\eps, \sigma^\eps$ approximating $a,\sigma$ in a suitable sense. The control terms $\chi^\eps$ will be chosen in a way to guarantee that \eqref{SDDE-Y-eps} has unique \emph{strong} solutions, for any $t\geq 0$
\be\label{w1-conv}
\bY_t^\eps\to \bX_t, \quad \eps\to 0
\ee
in probability, and  for any $T$
\be\label{w2-RN}
\E \int_0^T|\chi^\eps(t)|^2\, \di t\to 0, \quad \eps\to 0.
\ee
If $\sigma$ is non-degenerate, the latter relation yields that the law of $Y^\eps$ on each finite time interval is asymptotically close to the law of the \emph{strong} solution to
\be\label{SDDE-X-eps}
\di X^\eps(t)=a^\eps(\bX_t^\eps)\, \di t +\sigma^\eps( \bX_t^\eps)\, \di W(t),\quad t\ge0.
\ee
Hence the laws of $\bX^\eps$ are uniquely defined, and these  laws weakly converge to $\bX$; this yields weak uniqueness for \eqref{SDDE}.
The above argument actually exploits  the same  C-n-R strategy: thanks to the control term the lack of Lipschitz continuity of the  coefficients is compensated, and at the reimbursement stage we change the (controlled) process  $\bY^\eps$ to the (non-controlled) process $ \bX^\eps.$

\section{Proof of Theorem \ref{t1}}\label{s4}

\subsection{Existence of a weak solution}\label{s41} Existence of a weak solution can be established in a quite standard way, based on a compactness argument. Both for this purpose and for the subsequent proof of weak uniqueness, we  fix families $\{a^\eps\}, \{\sigma^\eps\}$   such that
\begin{itemize}
  \item[(i)] $a^\eps\to a, \sigma^\eps\to \sigma, \eps\to 0$ uniformly on each compact subset of $\CC$;
  \item[(ii)] conditions $\mathbf{H}_1 - \mathbf{H}_4$ hold true for $a^\eps, \sigma^\eps$ uniformly in $\eps$; that is, with constants which do not depend on $\eps$.
  \item[(iii)] the functions $a^\eps, \sigma^\eps$ are Lipschitz continuous on each bounded subset of $\CC$.
\end{itemize}

Note that   such a family is easy to construct. Namely, one can consider a family $P^\eps$ of finite-dimensional projectors in $\CC$ which strongly converge to the identity and such that $P^\eps x(0)=x(0)$, $x\in \CC$. Then $\wt a^\eps(x)=a(P^\eps x), \wt \sigma^\eps(x)=\sigma(P^\eps x)$ satisfy (i), (ii), and now $\wt a^\eps, \wt \sigma^\eps$ are essentially finite-dimensional. Taking convolutions with finite-dimensional approximate $\delta$-functions one obtains the required families $\{a^\eps\}, \{\sigma^\eps\}$.

By property (iii) and \eqref{gba}, \eqref{gbsigma}, equation \eqref{SDDE-X-eps} with the initial condition $\bX^\eps_0=x$ has a unique strong solution. By  It\^o's formula and \eqref{gba}, \eqref{gbsigma}, for any $p\geq 2$ there exists some  $C_p$ such that
$$
\di|X^\eps(t)|^p=\xi^{\eps, p}(t)\, \di t+\eta^{\eps,p}(t)\, \di W(t)
$$
with
$$
\xi^{\eps, p}(t)\leq C_p(1+\|\bX^\eps_t\|^p), \quad |\eta^{\eps,p}(t)|\leq C_p(1+\|\bX^\eps_t\|^p).
$$
We have
$$
|X^\eps(t)|^p\leq |x(0)|^p+ \int_0^t\big(\xi^{\eps, p}(s)\big)_+\, \di s+\int_0^t\eta^{\eps,p}(s)\, \di W(s),
$$
and thus
$$
\sup_{\tau\in [0, t]}|X^\eps(\tau)|^p\leq |x(0)|^p+ \int_0^t\big(\xi^{\eps, p}(s)\big)_+\, \di s+\sup_{\tau\in [0, t]}\left|\int_0^\tau\eta^{\eps,p}(s)\, \di W(s)\right|.
$$
Then by Cauchy's inequality and Doob's inequality,
$$
\E \sup_{\tau\in [0, t]}|X^\eps(t)|^{2p} \leq 3|x(0)|^{2p}+3C_p^2\E\left(t+\int_0^t \|\bX^\eps_s\|^p\, \di s\right)^2+ 12C_p^2\E\left(t+\int_0^t \|\bX^\eps_s\|^{2p}\, \di s\right).
$$
Note that
$$
\|\bX^\eps_s\|\leq \sup_{\tau\in [0, t]}|X^\eps(t)|+\|x\|, \quad s\leq t.
$$
Hence, by the Gronwall inequality, we get the bound
\be\label{p-uni}
\sup_{t\leq T, \eps>0}\E \sup_{\tau\in [0, t]}|X^\eps(t)|^{2p}<\infty, \quad T>0 ,\quad p\geq 2.
\ee
Denote
$$
\tau^\eps_R=\inf\{t: |X^\eps(t)|\geq R\},
$$
then it follows from \eqref{p-uni} that for any $T$
$$
\sup_\eps \P(\tau^\eps_R<T)\to 0, \quad R\to \infty.
$$
Recall the  coefficients $a^\eps, \sigma^\eps$ are bounded (uniformly in $\eps$) on each bounded subset in $\CC$. Then it is a standard routine to show that, for any $ \nu<1/2, q>0,$ and $T$ there exists $Q$ such that
$$
\sup_\eps\P(X^\eps|_{[0,T]}\not\in B^Q_\nu(0,T))\leq q,
$$
where
$$
B^Q_\nu(0,T)=\Big\{z: |z(t)|\leq Q,  |z(t)-z(s)|\leq Q|t-s|^\nu, \quad s,t\in [0,T]\Big\}
$$
is a ball in the space of $\nu$-H\"older continuous functions on $[0,T]$. This yields that the family of laws of $X^\eps, \eps>0$ in $C(\ax, \Re^n)$  are weakly compact. Since $a^\eps\to a, \sigma^\eps\to \sigma$ uniformly on compacts in $\CC$ and $a, \sigma$ are continuous,
any weak limit point for $X^\eps, \eps\to 0$ is a weak solution to \eqref{SDDE}; this argument is again quite standard, and thus we omit the details. This completes the proof of statement 1 of Theorem \ref{t1}.

\subsection{Weak uniqueness}\label{s42} As explained in Section \ref{s3}, we will specify the law of an arbitrary weak solution $X$ to \eqref{SDDE} with $\bX_0=x^0\in \CC$ as the weak limit of the laws of solutions to \eqref{SDDE-X-eps}.
For the given weak solution $X$ with $\bX_0=x^0\in \CC$, let $W$ be the corresponding Wiener process on a filtered probability space $(\Omega, \Ff, \{\Ff_t\}, \P)$.

We fix some  positive $\gamma<\alpha\wedge(2\beta-1)$. We take $Q>0$ (a free parameter, whose value will be specified later) and denote by
$B^{Q, x^0}_{1/3}(-r,0)$ the set of $x\in \CC$ such that for some $v\leq r$
$$
x(t)=x^0(t), \quad t\in [-r, -r+v], \quad x|_{[-r+v,0]}\in B^{Q, x^0}_{1/3}(-r+v,0).
$$

We define
$$
K_Q=B^{Q, x^0}_{1/3}(-r,0)\cup B^{Q}_{1/3}(-r,0),
$$
which is a compact set in $\CC$, and put
$$
\upsilon_\eps=\left(\sup_{x\in K_Q}|a^\eps(x)-a(x)|\right)^{1/\alpha}\vee\left(\sup_{x\in K_Q}\norm{\sigma^\eps(x)-\sigma(x)}\right)^{1/\beta}.
$$
Then
\be\label{comparison}
\sup_{x\in K_Q}|a^\eps(x)-a(x)|\leq \upsilon_\eps^\alpha, \quad \sup_{x\in K_Q}\norm{\sigma^\eps(x)-\sigma(x)}\leq \upsilon_\eps^\beta,
\ee
and we have $\upsilon_\eps\to 0$. In particular, there exists $\eps_0>0$ such that $|\upsilon_\eps|\leq 1$ for $\eps\leq \eps_0$. In what follows we consider $\eps\leq \eps_0$ only.

Consider a family of processes
$Y^\eps$ defined by the SDDEs
\be\label{SDDE-X-eps-ups}
\di Y^\eps(t)=a^\eps(\bY_t^\eps)\, \di t +\sigma^\eps( \bY_t^\eps)\, \di W(t)+\upsilon^{\gamma-1}_\eps (X(t)-Y^\eps(t))1_{t\leq \tau_\eps}\,\di t, \quad \bY_0^\eps=x^0,
\ee
where
\be\label{tau}
\tau_\eps=\inf\{s: |X(s)-Y^\eps(s)|\geq \upsilon_\eps\}.
\ee
Since $a^\eps, \sigma^\eps$ are Lipschitz continuous, the processes $Y^\eps$ are well defined. Define also
$$
\theta_Q=\inf\{s: X|_{[0,s]}\not\in B_{1/3}^Q(0,s)\}
$$
and observe that the calculation from the previous section yields
\be\label{theta}
\theta_Q\to \infty, \quad Q\to \infty
\ee
with probability 1.

 We have  by It\^o's formula
 \be\label{Ito}
 |Y^\eps(t)-X(t)|^2=\int_{0}^tA^\eps(s)\, \di s + \int_{0}^t\Sigma^\eps(s)\, \di W(s), \quad t\geq 0,
 \ee
 where
 $$\ba
 A^\eps(s)=2\Big(a^\eps(\bY_s^\eps)-a(\bX_s), Y^\eps(s)-X(s)\Big)
 &+\norm{\sigma^\eps(\bY_s^\eps)-\sigma(\bX_s)}^2-2\upsilon^{\gamma-1}_\eps \Big|Y^\eps(s)-X(s)\Big|^2 1_{s\leq \tau_\eps},
 \ea
 $$
 $$
 \Sigma^\eps(s)=2(Y^\eps(s)-X(s))^\top\Big(\sigma^\eps(\bY_s^\eps)-\sigma(\bX_s)\Big).
 $$
 By the Cauchy inequality,
 $$
 \norm{\sigma^\eps(\bY_s^\eps)-\sigma(\bX_s)}^2\leq 2\norm{\sigma^\eps(\bY_s^\eps)-\sigma^\eps(\bX_s)}^2+2\norm{\sigma^\eps(\bX_s)-\sigma(\bX_s)}^2.
 $$
  Then $A^\eps(s)\leq A^{\eps,1}(s)+A^{\eps,2}(s) $ with
 $$\ba
 A^{\eps,1}(s)&=2\Big(a^\eps(\bX_s)-a (\bX_s), Y^\eps(s)-X(s)\Big)
 +2\norm{\sigma^\eps(\bX_s)-\sigma(\bX_s)}^2,
 \\
 A^{\eps,2}(s)&=2\Big(a^\eps(\bY_s^\eps)-a^\eps(\bX_s), Y^\eps(s)-X(s)\Big)+2\norm{\sigma^\eps(\bY_s^\eps)-\sigma^\eps(\bX_s)}^2-2\upsilon^{\gamma-1}_\eps \Big|Y^\eps(s)-X(s)\Big|^2 1_{s\leq \tau_\eps}.
 \ea
 $$
 Note that
\be\label{bounds}
\bX_t\in B^Q_{1/3}(-r,0), \quad t\leq \theta_Q, \quad |Y^\eps(t)-X(t)|\leq \upsilon_\eps, \quad t\leq \tau_\eps,
\ee
which simply gives
\be\label{b2}\ba
 |A^{\eps,1}(s)|&\leq 2\sup_{x\in B^Q_{1/3}(-r,0)}\Big(|a^\eps(x)-a(x)|\upsilon_\eps+\norm{\sigma^\eps(x)-\sigma(x)}^2\Big)
\leq C\Big(\upsilon_\eps^{\alpha+1}+\upsilon_\eps^{2\beta}\Big), \quad s\leq \tau_\eps\wedge \theta_Q,
 \ea
\ee
see \eqref{comparison}. The second inequality in \eqref{bounds} clearly yields
$$
\|\bY^\eps_t-\bX_t\|\leq \upsilon_\eps, \quad t\leq \tau_\eps.
$$
Recall that $a^\eps, \sigma^\eps$  satisfy analogues of \eqref{H_gamma}, \eqref{H_delta} uniformly in $\eps$, and $\upsilon_\eps<1$. Then
\be\label{b3}
A^{\eps,2}(s)\leq C_a\upsilon_\eps^\alpha|Y^\eps(s)-X(s)|-2\upsilon_\eps^{\gamma-1}|Y^\eps(s)-X(s)|^2+C\upsilon_\eps^{2\beta}, \quad s\leq \tau_\eps.
\ee
By the Cauchy inequality,
$$
C_a\upsilon_\eps^\alpha|Y^\eps(s)-X(s)|\leq {C_a\over 2}\Big(\upsilon_\eps^{\alpha+1}+\upsilon_\eps^{\alpha-1}|Y^\eps(s)-X(s)|^2\Big).
$$
On the other hand, since $\gamma<\alpha$ and $\upsilon_\eps\to 0$, we can choose $\eps_1\in (0, \eps_0]$ such that
$$
{C_a\over 2}\upsilon_\eps^{\alpha-1}-2\upsilon_\eps^{\gamma-1}\leq -\upsilon_\eps^{\gamma-1}, \quad \eps\leq \eps_1.
$$
In what follows, we consider $\eps\leq \eps_1$ only. For such $\eps$, we get by \eqref{b2}, \eqref{b3}
\be\label{bA}
A^{\eps}(s)\leq -\upsilon_\eps^{\gamma-1}|Y^\eps(s)-X(s)|^2+C\Big(\upsilon_\eps^{\alpha+1}+\upsilon_\eps^{2\beta}\Big), \quad s\leq \tau_\eps\wedge \theta_Q,
\ee

A similar argument applies to $\Sigma^\eps(s)=\Sigma^{\eps,1}(s)+\Sigma^{\eps,2}(s),$
$$
\Sigma^{\eps,1}(s)=2(Y^\eps(s)-X(s))^\top(\sigma^\eps(\bX_s)-\sigma(\bX_s)), \quad \Sigma^{\eps,2}(s)=2(Y^\eps(s)-X(s))^\top(\sigma^\eps(\bY_s^\eps)-\sigma^\eps(\bX_s)).
$$
Likewise to  \eqref{b2},
\be\label{b2_bis}
 |\Sigma^{\eps,1}(s)|\leq 2\upsilon_\eps\sup_{x\in B^Q_{1/3}(-r,0)}\norm{\sigma^\eps(x)-\sigma(x)}, \quad s\leq \tau^\eps\wedge \theta_Q,
\ee
 and likewise  to \eqref{b3},
\be\label{b3_bis}
|\Sigma^{\eps,2}(s)|\leq C  \upsilon^{1+\beta}_\eps, \quad s\leq \tau^\eps,
\ee
which gives
\be\label{bSigma}
|\Sigma^{\eps}(s)|\leq C  \upsilon^{1+\beta}_\eps, \quad s\leq \tau^\eps \wedge \theta_Q.
\ee

We are going to  apply Lemma \ref{lem1} with a fixed $T>0$ and
$$
U(t)=|Y^\eps(t)-X(t)|^2, \quad \tau= \tau^\eps \wedge \theta_Q\wedge T.
$$
By \eqref{bA}, \eqref{bSigma}  the assumptions of Lemma \ref{lem1} hold with
$$
\lambda=\upsilon_\eps^{\gamma-1}, \quad A=C\Big(\upsilon_\eps^{\alpha+1}+\upsilon_\eps^{2\beta}\Big), \quad B= C\upsilon_\eps^{2+2\beta}.
$$
Recall that $\gamma<\alpha$, hence there exists $\chi>0$ such that
\be\label{asympA}
A\lambda^{-1}= C\Big(\upsilon_\eps^{\alpha+2-\gamma}+\upsilon_\eps^{2\beta+1-\gamma}\Big)\leq \upsilon_\eps^{2+\chi}
\ee
for $\eps>0$ small enough. Next, we have
$$
B^{1/2}\lambda^{-\delta}= C\upsilon_\eps^{1+\beta+\delta(1-\gamma)},
$$
and
$$
1+\beta+{1\over 2}(1-\gamma)=2+{2\beta-1-\gamma\over 2}>2.
$$
That is, we can fix $\delta<1/2$ close enough to $1/2$ and then choose $\chi>0$ small enough such that, in addition to \eqref{asympA},
$$
B^{1/2}\lambda^{-\delta}\leq \upsilon_\eps^{2+2\chi}
$$
for $\eps>0$ small enough. Denote
$$
H_\eps=\Big\{\sup_{t\leq \tau^\eps \wedge \theta_Q\wedge T}|Y^\eps(t)-X(t)|^2\geq 2\upsilon_\eps^{2+\chi}\Big\}.
$$
Then by Lemma \ref{lem1} with $R=\upsilon_\eps^{-\chi}$ we have
\be\label{H_neg}
\P\left(H_\eps\right)\leq C_1 e^{-C_2  \upsilon_\eps^{-2\chi}}\to 0, \quad \eps\to 0.
\ee
On the other hand, clearly
$$
\upsilon_\eps^{2+\chi}=o(\upsilon_\eps^2),\quad \eps\to 0
$$
and thus for $\eps$ small enough we have
$$
\tau^\eps \wedge \theta_Q\wedge T=\theta_Q\wedge T
$$
on $\Omega\setminus H_\eps$. This yields that, for $\eps>0$ small enough,
\be\label{final_bound}
\sup_{t\in [0,T]}|Y^\eps(t)-X(t)|\leq \upsilon_\eps\quad \hbox{on the set}\quad \{\theta_Q\geq T\}\setminus H_\eps.
\ee

 Now we can finish the proof. Let $T>0$ be fixed and $F$ be a bounded continuous function on $C([0, T],\Re^n)$. We have
 $$
 |\upsilon^{\gamma-1}_\eps (X(t)-Y^\eps(t))1_{t\leq \tau_\eps}|\leq \upsilon_\eps^\gamma
 $$
 with positive $\gamma$ and $\upsilon_\eps\to 0.$ This immediately gives \eqref{w2-RN}, which leads to the bound similar to \eqref{TV}:
$$
d_{TV}\Big(\mathrm{Law}(Y^\eps|_{[0, T]}), \mathrm{Law}(X^\eps|_{[0, T]})\Big)\to 0, \quad \eps\to 0.
 $$
 Since $F$ is bounded, this gives
  $$
 \E F(Y^\eps|_{[0, T]})-\E F(X^\eps|_{[0, T]})\to 0, \quad \eps\to 0.
 $$
 On the other hand, it follows from \eqref{H_neg}, \eqref{final_bound} that, on the set $\{\theta_Q\geq T\}$,
 $$
 Y^\eps|_{[0, T]}\to X|_{[0, T]}, \quad \eps\to 0
 $$
in probability in $C([0, T],\Re^n)$. Then
$$
\limsup_{\eps\to 0}|\E F(Y^\eps|_{[0, T]})-\E F(X|_{[0, T]})|\leq 2\sup_{x}|F(x)|\P(\theta_Q\geq T).
$$
Combining these two inequalities, we get
\be\label{w_limsup}
\limsup_{\eps\to 0}|\E F(X^\eps|_{[0, T]})-\E F(X|_{[0, T]})|\leq 2\sup_{x}|F(x)|\P(\theta_Q\geq T).
\ee
Recall that the choice of $Q$ determines further details in the construction of the generalized coupling, e.g.  the choice of $\upsilon^\eps$. However,  \eqref{w_limsup} does not involve $Y^\eps$, and $Q>0$ therein is just a free parameter. Taking $Q\to \infty$ and using \eqref{theta}, we finally deduce that
$$
\E F(X^\eps|_{[0, T]})\to \E F(X|_{[0, T]}), \quad \eps\to 0.
$$
This completes the proof of weak uniqueness, since an arbitrary weak solution $X^\eps$ to \eqref{SDDE} is now uniquely specified on any finite time interval $[0, T]$ as the weak limit of the solutions to \eqref{SDDE-X-eps}.

\subsection{Continuity and the Markov property} Denote by $P_{t,x} ,t\geq 0, x\in \CC$ the law of $\bX_t$, where $X$ is the (unique in law) solution to \eqref{SDDE} with $\bX_0=x$. Denote by $\{P_{t,x}^\eps\}$ the corresponding laws for the approximating sequence $X^\eps$ defined by \eqref{SDDE-X-eps}, and consider the respective families of integral operators
$$
T_tf(x)=\int_{\CC}f(y)P_{t,x}(\di y), \quad T_t^\eps f(x)=\int_{\CC}f(y)P_{t,x}^\eps(\di y), \quad f\in C_b(\CC).
$$
We have just proved that, for a given $f\in C_b(\CC)$,
$$
 T_t^\eps f(x)\to  T_tf(x), \quad \eps\to 0
$$
point-wise. This convergence is actually uniform on each compact subset $K$ of $\CC$: this can be obtained using just a slight  modification of the previous proof, where the H\"older ball $B^Q_{1/3}(-r,0)$ in the choice of $\upsilon_\eps$ is replaced by  $B^Q_{1/3}(-r,0)\cup K$;
we omit the details.  Then the functions $T_tf, t\geq 0, f\in C_b(\CC)$ are  continuous and bounded.

 Now the Markov property for $\bX$ is obtained from the same property for $\bX^\eps$ by the usual approximation argument: for arbitrary $t>s>s_1, \dots s_k$, $f\in C_b(\CC)$ and $G\in C_b(\CC^{k+1})$, we have
$$\ba
\E f(\bX_t)G(\bX_s, \dots, \bX_{s_k})&=\lim_{\eps\to 0}\E f(\bX_t^\eps)G(\bX_s^\eps, \dots, \bX_{s_k}^\eps)
\\&=\lim_{\eps\to 0}\E T_{t-s}^\eps f(\bX_s^\eps)G(\bX_s^\eps, \dots, \bX_{s_k}^\eps)=\E T_{t-s}f(\bX_s)G(\bX_s, \dots, \bX_{s_k});
\ea
$$
in the last identity we use  that $X^\eps\to X$ weakly and $T_{t-s}^\eps f\to T_{t-s}f$ uniformly on compacts. This proves that
$\bX$ is a time-homogeneous Markov process with the transition function $\{P_{t,x}(\di y)\}$. The Feller property has  already been proved: for $f\in C_b(\CC)$, the functions $T_tf, t\geq 0$ also belong to $C_b(\CC)$.

\section{Proofs of Theorem \ref{t2} and Theorem \ref{t23}}\label{s5} Let us give a short outline.
We will prove Theorem \ref{t2} in two steps. First, we will show that for $N$ large enough,  $d_{N, \gamma}$ is contracting on the set
\be\label{near-the-diagonal}
D_{N, \gamma}=\{(x,y): d_{N, \gamma}(x,y)<1\}.
\ee
Since $d_{N, \gamma}(x,y)\leq 1$ everywhere, this will immediately yield that $d_{N, \gamma}$ is non-expanding.  Then we will prove the following support-type statement: for any given $z\in \CC$, $\delta>0$, and $h >r$,
\be\label{support}
\inf_{x\in B}\P_x(\|\bX_h-z\|\leq \delta)>0,
\ee
where $B$ is either a bounded set, or $B=H_c$ and \eqref{gba1}, \eqref{gbsigma1} hold true. These two principal statements combined will complete the proof of Theorem \ref{t2}; for convenience of the reader we prove them separately in Sections \ref{s51} and \ref{s52} below. In Section \ref{s53} we prove Theorem \ref{t23} as a corollary of Theorem \ref{t2} and the general theory.

\subsection{Contraction property of  $d_{N, \gamma}$ on $D_{N, \gamma}$.}\label{s51} The proof is based on the generalized coupling construction, very similar to the one introduced in Section \ref{s42}. Recall that the coefficients of \eqref{SDDE} are not Lipschitz continuous. To overcome this minor difficulty, we will systematically use the following trick: first, we make the construction for Lipschitz continuous coefficients; then, we provide estimates for the generalized coupling, which involve the constants from the conditions $\mathbf{H}_1$ -- $\mathbf{H}_4$ only; finally, we remove the additional assumption for the coefficients to be  Lipschitz continuous by an approximation argument.

For given $x,y\in \CC$, take $\upsilon=\|x-y\|$. Let $X$ be the solution to \eqref{SDDE} with $\bX_0=x$ and $Y$ be the solution to the following controlled equation, similar to \eqref{SDDE-X-eps-ups}:
$$
\di Y(t)=a(\bY_t)\, \di t +\sigma(\bY_t)\, \di W(t)+\upsilon^{\gamma-1} (X(t)-Y(t))1_{t\leq \tau}\,\di t, \quad \bY_0^\eps=y,
$$
$$
\tau=\inf\{s: |X(s)-Y(s)|\geq 2\upsilon\}.
$$
At the moment,   we  additionally assume $a,\sigma$ to be Lipschitz continuous, hence  $X,Y$ are  well defined as the strong solutions to the corresponding equations.

We repeat, in a slightly different and actually simpler setting, the calculations from Section \ref{s42}. Namely,  by It\^o's formula,
$$
 |Y(t)-X(t)|^2=\upsilon^2+\int_{0}^tA(s)\, \di s + \int_{0}^t\Sigma(s)\, \di W(s), \quad t\geq 0,
$$
 $$\ba
 A(s)=2\Big(a(\bY_s)-a(\bX_s), Y(s)-X(s)\Big)
 &+\norm{\sigma(\bY_s)-\sigma(\bX_s)}^2-2\upsilon^{\gamma -1} \Big|Y(s)-X(s)\Big|^21_{s \le \tau},
 \ea $$
 $$
 \Sigma(s)=2(Y(s)-X(s))^\top\Big(\sigma(\bY_s)-\sigma(\bX_s)\Big).
 $$
 We will take $\|x-y\|=\upsilon\leq \upsilon_0$ with $\upsilon_0>0$ small enough. In particular, we will have $2\upsilon\leq 1$, which allows one to apply \eqref{H_gamma}, \eqref{H_delta} and get similarly to \eqref{bA}, \eqref{bSigma}
\be\label{bASigma}
 A(s)\leq -\upsilon^{\gamma-1}|Y(s)-X(s)|^2+C\Big(\upsilon^{\alpha+1}+\upsilon^{2\beta}\Big), \quad |\Sigma(s)|\leq C  \upsilon^{1+\beta},
 \quad s\leq \tau.
 \ee
 We take $\chi, \delta, \kappa$ the same as in Section \ref{s42}, and  use Lemma \ref{lem1} and the same calculation as in the proof of \eqref{final_bound}. We obtain  that, for a given $T>0$, there exist $\upsilon_0>0$ and $C_1, C_2>0$ such that for $\|x-y\|=\upsilon\leq \upsilon_0$
\be\label{final_bound2}
 \P\left(\sup_{t\leq T}\Big(|Y(t)-X(t)|^2- e^{-\upsilon^{\gamma-1} t}\|x-y\|^2\Big)\geq\upsilon^{2+\chi}+\upsilon^{(4+\chi)\delta-\kappa}\right)\leq C_1 e^{-C_2  \upsilon^{-2\kappa}}.
 \ee
 Note that the choice of the indices $\gamma, \chi, \delta, \kappa$ in the above construction is determined by the H\"older indices $\alpha, \beta$; once this choice is fixed, the level  $\upsilon_0$ and the constants $C_1, C_2$ depend only on $T$ and the constants from the assumptions \eqref{H_gamma}, \eqref{H_delta}.

 Now, let $h>r$ be fixed. The inequality
$$
 \sup_{t\leq h}\Big(|Y(t)-X(t)|^2- e^{-\upsilon^{\gamma-1} t}\|x-y\|^2\Big)\leq\upsilon^{2+\chi}+\upsilon^{(4+\chi)\delta-\kappa}
$$
 yields the bound
 $$
 \|\bX_h-\bY_h\|\leq\left(\upsilon^{-2}e^{-\upsilon^{\gamma-1} (h-r)}+\upsilon^{2+\chi}+\upsilon^{(4+\chi)\delta-\kappa}\right)^{1/2}\|x-y\|.
 $$
 Clearly,
 $$
 \upsilon^{-2}e^{-\upsilon^{\gamma-1} (h-r)}+\upsilon^{2+\chi}+\upsilon^{(4+\chi)\delta-\kappa}\to 0, \quad \upsilon\to 0.
 $$
 Then, taking $\upsilon_0$ we obtain finally that there exists $\upsilon_0>0$ such that
 \be\label{pre-contraction}
 \P\left(\|\bX_h-\bY_h\|\geq  {1\over 2}\|x-y\|\right)\leq C_1 \exp(-C_2  \|x-y\|^{-2\kappa}), \quad \|x-y\|\leq \upsilon_0.
 \ee
 On the other hand,
 the control term
 $$
 \chi(t)=\upsilon^{\gamma-1} (X(t)-Y(t))1_{t\leq \tau}
 $$
 satisfies
 $\chi(t)\leq 2\|x-y\|^{\gamma}, t\geq 0.$ Similarly to \eqref{TV}, from this and $\mathbf{H}_3$ we get that
\be\label{TV1}
d_{TV}\Big(\mathrm{Law}(\bY_h), P_y^h\Big)\leq C_3\|x-y\|^{\gamma},  \quad x,y\in \CC.
\ee

Now it is easy to perform the ``reimbursement'' step; that is, to derive a required  bound for
$\E d_{N,\gamma}(\bX_h, \wh\bY_h)$, where   $X, \wh Y$ is a properly constructed  (true) coupling. For the reader's convenience we formulate this step in a separate proposition, which is a  modification of statement (i) of \cite[Theorem 2.4]{BKS18}.

\begin{prop}\label{p_coup} Let, for a family $\{\mu^x, x\in \CC\}\subset \mathcal{P}(\CC)$, the families of $\CC$-valued random elements $\{\xi^{x,y},  x,y\in \CC\}$, $\{\eta^{x,y}, x,y\in \CC\}$ be given such that
\begin{itemize}
 \item[(i)] $\mathrm{Law}(\xi^{x,y})=\mu^x,$ and for some $\gamma\in (0,1], \upsilon>0,  C>0$,
  $$
  d_{TV}\Big(\mathrm{Law}(\eta^{x,y}), \mu^y\Big)\leq C\|x-y\|^{\gamma},  \quad \|x-y\|\leq \upsilon;
  $$
  \item[(ii)] for some $\theta\in (0,1)$, and a function $p(s)=o(s^\gamma), s\to 0$, $p(s)\leq 1$,
  $$
  \P\left(\|\xi^{x,y}-\eta^{x,y}\|> \theta\|x-y\|\right)\leq p(\|x-y\|), \quad \|x-y\|\leq \upsilon.
  $$

\end{itemize}

Then for any $\theta_1\in  (\theta^\gamma, 1)$ there exists $N_0=N_0(\gamma, \upsilon, C, \theta, \theta_1, p(\cdot))$ such that for $N\geq N_0$
$$
d_{N, \gamma}(\mu^x, \mu^y)\leq \theta_1 d_{N, \gamma}(x,y)
$$
on the set $\{(x,y):d_{N, \gamma}(x,y)<1\}$.
\end{prop}
\begin{proof} Take $N_1=\upsilon^{-\gamma}$, then for $N\geq N_1$
$$
d_{N, \gamma}(x,y)<1\quad \Leftrightarrow\quad  N\|x-y\|^\gamma<1\quad \Rightarrow\quad \|x-y\|<\upsilon.
$$ In what follows we take $N\geq N_1$ and $x,y$ such that $d_{N, \gamma}(x,y)<1$; then (i) and (ii) hold true.

The following useful fact is well known (\cite[Problem 11.8.8]{Dudley}, see also \cite[Lemma 4.3.2]{K17}): if $(\xi, \eta)$ and $(\xi',\eta')$ are two pairs of random elements valued in a Borel measurable space, such that $\eta$ and $\xi'$ have the same distribution, then on  a properly chosen probability space there exists  a triple of  random elements $\zeta_1$, $\zeta_2$, $\zeta_3$ such that the law of $(\zeta_1, \zeta_2)$  coincides with the law of $(\xi, \eta)$ and the law of $(\zeta_2, \zeta_3)$   coincides with the law of $(\xi', \eta')$. On the other hand, by the assumption (i) and the Coupling lemma  (e.g. \cite[Section 1.4]{Thor} or \cite[Theorem 2.2.2]{K17}), on a properly chosen probability space there exists a pair of random elements $\xi', \eta'$ such that $\mathrm{Law}(\xi')=\mathrm{Law}(\eta^{x,y}),$ $\mathrm{Law}(\eta')=\mu^y, $ and
$$
\P(\xi'\not =\eta')=d_{TV}\Big(\mathrm{Law}(\eta^{x,y}), \mu^y\Big)\leq C\|x-y\|^{\gamma};
$$
here we adopt the definition
$$
d_{TV}(\mu, \nu)=\sup_{A}|\mu(A)-\nu(A)|.
$$
Take $\xi=\xi^{x,y}, \eta=\eta^{x,y}$ and consider the corresponding triple $\zeta_1, \zeta_2, \zeta_3$. Then $\zeta_1, \zeta_3$ is a (true) coupling for $\mu^x, \mu^y$, and
$$
  \P\left(\|\zeta_1-\zeta_2\|\geq \theta\|x-y\|\right)\leq p(\|x-y\|), \quad  \P(\zeta_2\not =\zeta_3)\leq C\|x-y\|^{\gamma}.
$$
Recall that $d_{N, \gamma}\leq 1$, hence
$$\ba
\E d_{N, \gamma}(\zeta_1, \zeta_3)&\leq \E d_{N, \gamma}(\zeta_1, \zeta_2)+\P(\zeta_2\not =\zeta_3  )
\\&\leq \E d_{N, \gamma}(\zeta_1, \zeta_2)1_{\|\zeta_1-\zeta_2\|\leq \theta\|x-y\|}+\P\left(\|\zeta_1-\zeta_2\|\geq \theta\|x-y\|\right)+\P(\zeta_2\not =\zeta_3  ).
\ea
$$
Recall that $d_{N, \gamma}(x,y)<1$, hence
$$
\E d_{N, \gamma}(\zeta_1, \zeta_2)1_{\|\zeta_1-\zeta_2\|\leq \theta\|x-y\|}\leq N\theta^\gamma \|x-y\|^\gamma,
$$
and
$$
N \|x-y\|^\gamma= d_{N, \gamma}(x,y).
$$
Then
$$\ba
d_{N, \gamma}(\mu^x, \mu^y)&\leq \E d_{N, \gamma}(\zeta_1, \zeta_3)\leq N\theta^\gamma \|x-y\|^\gamma+p(\|x-y\|)+C\|x-y\|^\gamma
\\&\leq \left(\theta^\gamma +{1\over N}\|x-y\|^{-\gamma}p(\|x-y\|)+{C\over N}\right) d_{N, \gamma}(x,y).
\ea
$$
Since $p(s)\leq 1$ and $s^{-\gamma}p(s)\to 0,s\to 0$, we have that
$$
C_p:=\sup_{s>0}s^{-\gamma}p(s)<\infty.
$$
Define $N_2$  by the identity
$$
\theta^\gamma {C_p+C\over N_2}=\theta_1.
$$
Then the required statement holds true for $N_0=\max(N_1, N_2)$.
\end{proof}

Now, we can complete the proof of the contraction property of  $d_{N, \gamma}$ for $\bX^h$ on $D_{N, \gamma}$. Let $\gamma, \kappa, \upsilon_0, C_1, C_2, C_3$ be the same as in \eqref{pre-contraction} and \eqref{TV1}. We apply Proposition \ref{p_coup} with
$$
\upsilon=\upsilon_0, \quad C=C_3, \quad p(s)=\Big(C_1 \exp(-C_2  s^{-2\kappa})\Big)\wedge 1, \quad \theta= 2^{-1}, \quad \theta_1=2^{-1}(1+2^{-\gamma}),
$$
and obtain that there exists $N_0$ such that
\be\label{contracting3}
d_{N, \gamma}(P_x^h, P_y^h)\leq (2^{-1}+2^{-1-\gamma}) d_{N, \gamma}(x,y), \quad (x,y)\in D_{N, \gamma}, \quad N\geq N_0,
\ee
which provides the required  contraction property under  the additional assumption that $a, \sigma$ are Lipschitz continuous.

The last step in the proof is to remove  this limitation; for that, we use an approximation procedure. The choice of the index $\gamma$ and the constant $N_0$ in  \eqref{contracting3} is determined only by the assumptions $\mathbf{H}_1 - \mathbf{H}_3.$  Let a family of processes $\{X^\eps\}$ be defined by \eqref{SDDE-X-eps} with  $a^\eps, \sigma^\eps$ same as in Section \ref{s41}, then the corresponding transition probabilities satisfy a uniform analogue of \eqref{contracting3}:
\be\label{contracting4}
d_{N, \gamma}(P_x^{h, \eps}, P_y^{h, \eps})\leq (2^{-1}+2^{-1-\gamma}) d_{N, \gamma}(x,y), \quad (x,y)\in D_{N, \gamma}, \quad N\geq N_0, \quad \eps>0.
\ee
We have already proved in Section \ref{s42} that $P_x^{h, \eps}\to P_x^{h}$ weakly as $\eps\to 0$. Note that $d_{N, \gamma}(x,y)$ is a bounded metric on $\CC$, and the convergence in this metric is the same as in the standard one. Hence weak convergence in $\mathcal{P}(\CC)$ is equivalent to convergence w.r.t. the  coupling distance $d_{N, \gamma}$. In particular,
$$
d_{N, \gamma}(P_x^{h, \eps}, P_y^{h, \eps})\to d_{N, \gamma}(P_x^{h}, P_y^{h}), \quad \eps\to 0,
$$
and thus \eqref{contracting3} follows from  \eqref{contracting4}.

\begin{rem} The generalized coupling construction, used in the proof above, can be  also used for a study of the continuous time family $P_x^t, t\geq 0$. Namely, using \eqref{final_bound2} in a similar way as in the proof of Proposition \ref{p_coup}, we get that there exists a constant $C_h$ such that
\be\label{lip}
d_{\gamma,N}(P_x^t, P_y^t)\leq C_hd_{\gamma,N}(x,y), \quad x,y\in \CC, \quad t\in [0,h].
\ee
\end{rem}

\begin{rem}\label{r52} There is another possibility,  not used in  the previous proof: instead of making the ``reimbursement step'' at the time segment $[0,h]$, one can  iterate the ``control'' step on the segments $[h,2h], [2h, 3h], \dots$. The corresponding pair of processes $\bX_t, \bY_t\geq 0$ will satisfy then \be\label{pre-contraction-l}
 \P\left(\|\bX_{lh}-\bY_{lh}\|\geq  {1\over 2^l}\|x-y\|\right)\leq C_1 \exp(-C_2 2^{2\kappa l}  \|x-y\|^{-2\kappa}), \quad \|x-y\|\leq \upsilon_0, \quad l\geq 1
 \ee
 by \eqref{pre-contraction}. This bound combined with the Markov property and \eqref{TV1} will give
  \be\label{TV2}\ba
d_{TV}\Big(\mathrm{Law}(\bY), P_y\Big)&\leq C_3\|x-y\|^{\gamma}
\\&+\sum_{l=1}^\infty\left(C_32^{-\gamma l}\|x-y\|^{\gamma}+C_1 \exp(-C_2 2^{2\kappa l}  \|x-y\|^{-2\kappa})\right),  \|x-y\|\leq \upsilon_0,
\ea
\ee
where $P_y$ denotes the law of $\bX_t, t \geq 0$ with $\bX_0=y$ in the path space $C([0, \infty), \CC)$, and $\mathrm{Law}(\bY)$ is understood in the same sense. That is, essentially the same construction as in the  above proof gives a generalized coupling for the entire path of the segment process. Making  now the ``reimbursement step'' similarly to (and simpler than) Proposition \eqref{p_coup}, one can construct a (true) coupling
$\bX_t, \bY_t, t \geq 0$ for $P_x, P_y$ such that
$$
\P(\|\bX_t- \bY_t\|\to 0, \, t\to \infty)\leq C_4\|x-y\|^\gamma.
$$
\end{rem}

\subsection{Proof of \eqref{support}.}\label{s52} We prove the support-type assertion \eqref{support} using a  stochastic control argument, which  is similar to, and simpler than,  the one from Section \ref{s42}, Section \ref{s51}. Thus we give the principal steps only,  and  omit the details.

Let for  a given $z\in \CC, h>r$ the function  $z^h\in C([0,h],\Re^n)$ be defined by
$$
z^h(t)=z(t-h), \quad t\in [h-r, h], \quad z^h(t)=z(h-r){t\over h-r}, \quad t\in [0,h-r].
$$
We consider a family of processes $X^{\lambda,x}, \lambda>0$ defined by
\be\label{SDDE-lambda}
\di X^{\lambda,x}(t)=a(\bX_t^{\lambda,x})\, \di t +\sigma(\bX_t^{\lambda,x})\, \di W(t)-\lambda \Big(X^{\lambda,x}(t)-z^h(t)\Big)\, \di t,\quad t\ge0, \quad \bX_0^{\lambda,x}=x.
\ee
Since we need these processes to be well defined, we assume for a while that $a, \sigma$ are  Lipschitz continuous. Then there exists $\lambda$ large enough such that
\be\label{support-lambda}
\inf_{x\in B}\P(\|\bX_h^{\lambda,x}-z\|\leq \delta)\geq {1\over 2};
\ee
recall that $B$ is either a bounded set, or $B=H_c$ and \eqref{gba1}, \eqref{gbsigma1} hold true. The proof of the above  statement is completely analogous to the proof of \eqref{final_bound2}, and is omitted. In addition, one has
\be\label{support-lambda2}
C_1(\lambda):=\sup_{x\in B}\E\sup_{t\in [0,h]}|X^{\lambda,x}(t)|^2<\infty.
\ee

On the other hand, $X^\lambda$ solves \eqref{SDDE} with $\di W(t)$ changed to
$$
\di W^{\lambda,x}(t)=\di W(t)-\lambda \sigma(\bX_t^\lambda)^{-1}X^{\lambda,x}(t)\, \di t.
$$
Since we  have assumed $a, \sigma$ to be Lipschitz continuous, for each $x\in \CC$ there exists a measurable mapping $\Phi_x: C([0, h], \Re^m)\to C([0, h], \Re^m)$ which resolves \eqref{SDDE} with the initial condition $\bX_0=x$ up to the time moment $h$; that is,
$$
X|_{[0,h]}=\Phi_x\Big(W|_{[0,h]}\Big).
$$
At the same time, we have
$$
X^{\lambda,x}|_{[0,h]}=\Phi_x\Big( W^{\lambda,x}|_{[0,h]}\Big).
$$
By Theorem \ref{tKLGirsanov} and \eqref{support-lambda2},
\be\label{KLsupport}
\sup_{x\in B}D_{KL}(\mathrm{Law}\, W^{\lambda,x}|_{[0,h]}\|\mathrm{Law}\, W|_{[0,h]})\leq {\lambda h\over 2}C_1(\lambda)\sup_{y\in \CC}\norm{\sigma(x)}=:C_2.
\ee
Now we apply \eqref{diffbound} with $\mu=\mathrm{Law}\, W^{\lambda,x}|_{[0,h]}$, $\nu=\mathrm{Law}\, W|_{[0,h]}$, and
$$
A=\{w\in C([0,h], \Re^m): \sup_{t\in [h-r,h]}|\Phi_x(w)(t)|\leq \delta\}
$$
to get for any $N>1$
$$
\inf_{x\in B} \P_x(\|\bX_h-z\|\leq \delta)\geq \frac1N \inf_{x\in B} \P(\|\bX_h^{\lambda,x}-z\|\leq \delta)-\frac{C_2+\log 2}{N\log N}\geq
\frac1{2N} -\frac{C_2+\log 2}{N\log N}.
$$
Taking $N=N_1=\exp(4C_2+4\log 2)$, we get
\be\label{support-final}
\inf_{x\in B} \P_x(\|\bX_h-z\|\leq \delta)\geq \frac1{4N_1},
\ee
 In the above construction $\lambda, C_1, C_2, N_1$ depend only on the constants from the assumption $\mathbf{H}_3$ and inequalities \eqref{gba}, \eqref{gbsigma},  the set $B$, and the time step $h$. That is, using the same approximation argument as in the previous section, we can get rid of the additional assumption that $a, \sigma$ are Lipschitz continuous.  This gives \eqref{support-final} without any extra assumptions, and completes the proof of \eqref{support}.

\subsection{Proof of Theorem \ref{t23}}\label{s53} Since $d_{\gamma, N}(x,y)$ decreases as a function of $\gamma$, without loss of generality we further assume that $\gamma$ satisfies the assumption $\gamma < \alpha\wedge (2\beta-1)$ from Theorem \ref{t2}. Fix $\ell$ such that
$$
\phi(1+\ell)>2C_V,
$$
where $\phi, C_V$ are respectively the function and the constant from the Lyapunov condition \eqref{lyap}.  Take $K_{V,\ell}=\{x:V(x)\leq \ell\}$, this set is either bounded under \eqref{level_sets}, or is contained in $H_c$ for $c$ large enough if  \eqref{level_sets1} holds.   Hence  by
Theorem \ref{t2} there exists $N$ such that $K_{V,\ell}$ is a $d_{N, \gamma}$-small set; that is, the condition \textbf{I}  of \cite[Theorem 4.5.2]{K17} holds true for
$$
d=d_{N, \gamma}, \quad B=K_{V,\ell}\times K_{V,\ell}.
$$
 On the other hand, by \cite[Theorem 2.8.6]{K17} the recurrence condition \textbf{R} (i),(ii) of \cite[Theorem 4.5.2]{K17} holds true with
$W(x,y)=V(x)+V(y),$ and  $\lambda(t)=\Phi^{-1}(t)$. We define
$$
\wh W(x,y)={\phi(V(x))+\phi(V(y))\over \phi(1)}, \quad \wh \lambda(t)={\phi(\Phi^{-1}(t))\over \phi(1)}={r(t)\over \phi(1)}
$$
and observe that the recurrence condition \textbf{R} (i),(ii)  of \cite[Theorem 4.5.2]{K17} with these functions and the same $B$  holds true, as well; see  the proof of \cite[Theorem 2.8.8]{K17}. In addition, by \cite[Proposition  2.8.5]{K17} applied to $U=V, V=\phi(V)$,  we have that
\be\label{Cesaro}
{1\over n}\E_x \sum_{k=1}^n \wh W(x, \bX_{kh})\leq \phi(V(x))+C_V+{1\over n}V(x), \quad n\geq 1.
\ee

Now we can obtain the required statement as a direct corollary of  \cite[Theorem 4.5.2]{K17}. Namely, take $q=\delta^{-1}$, $p=(1-\delta)^{-1}$, and denote
$$
d_{N,\gamma,p}(x,y)=d_{N,\gamma}(x,y)^{1/p}.
$$
This function is bounded by 1, hence \eqref{Cesaro} implies the additional assumption (4.5.8) in  \cite[Theorem 4.5.2]{K17}:
\be\label{Cesaro2}\ba
{1\over n}\E_x \sum_{k=1}^n d_{N,\gamma,p}(x,\bX_{kh}) \wh W(x, \bX_{kh})^{1/q}&
\leq \left({1\over n}\E_x \sum_{k=1}^n  \wh W(x, \bX_{kh})\right)^{1/q}\leq \left(\phi(V(x))+C_V+{1\over n}V(x)\right)^{1/q}.
\ea
\ee
Then  \cite[Theorem 4.5.2]{K17} yields that there exists unique IPM for $\bX$ and for some $c, C>0$
\be\label{rate_ind2}
d_{N,\gamma,p}(P_x^{nh}, \pi)\leq {C \over r(c n h)^{1/q}}\wh U(x), \quad x\in \CC, \quad n\geq 1
\ee
with
$$
\wh U(x)=\int_\CC d_{N,\gamma,p}(x,y) \wh W(x, y)^{1/q}\pi(\di y).
$$
By  Fatou's lemma and \eqref{Cesaro2},
$$
\wh U(x)\leq \left(\phi(V(x))+C_V\right)^{\delta}\leq  \phi(V(x))^\delta+C_V^{\delta}\leq C\phi(V(x))^\delta,
$$
in the last inequality we have used that $\inf_x\phi(V(x))>0$. By the Markov property and \eqref{lip} we have
$$
d_{N,\gamma,p}(P_x^{t+nh}, \pi)\leq C_h^{1/p} d_{N,\gamma,p}(P_x^{nh}, \pi),
$$
hence \eqref{rate_ind2} implies
\be\label{rate_ind3}
d_{N,\gamma,p}(P_x^{t}, \pi)\leq {C \over r(c t)^{\delta}}\phi(V(x))^\delta, \quad x\in \CC, \quad t\geq 0.
\ee
Since
$$
 d_\gamma(x,y)\leq d_{N,\gamma,p}(x,y),
$$
this completes the proof of \eqref{rate_ind}.

\section{Proofs of Theorems  \ref{t4} -- \ref{t6}. }\label{s6}

The proofs of all these three theorems will be based on the following auxiliary construction.  Denote
$$
\phi(v)=\mathrm{arctan}(|v|){v\over |v|}, \quad v\in \Re^n.
$$
Let $x\in \CC$ be arbitrary but fixed.  Denote by $X^x$ the (strong) solution to the SDDE \eqref{SDDE} with $\bX_0^x=x$, and consider a family of processes $Y^{\lambda, y}, \lambda\geq 0, y\in \CC$ defined  as the solutions to
\be\label{SDDE-Y-lambda}
\di Y^{\lambda,y}(t)=a(\bY_t^{\lambda,y})\, \di t +\sigma(\bY_t^{\lambda,y})\, \di W(t)-\lambda\phi(Y^{\lambda,y}(t)-X^x(t))\di t,\quad \bY_0^{\lambda,y}=y.
\ee
Note that each $Y^{\lambda,y}$ depends also on $x$, but we do not indicate this explicitly in order to keep the notation easy to read.
Denote
$$
\beta^{\lambda,y}(t)=\lambda\sigma(\bY^{\lambda,y}_t)^{-1} \phi(Y^{\lambda,y}(t)-X^x(t)),
$$
and observe that $|\beta^{\lambda,y}(t)|\leq C$. Then
the classical Girsanov theorem applies; e.g. \cite[Chapter 7]{LS}. Namely,
the family $$
\mathcal{E}^{\lambda,y}(t)=\exp\left(\int_0^t \beta^{\lambda,y}(s)\, \di W_s-{\frac12}\int_0^t|\beta^{\lambda,y}(s)|^2\, \di s\right), \quad y\in \CC
$$
satisfies $\E \mathcal{E}^{\lambda,y}(t)=1$, and the process
$$
W^{\lambda,y}(\tau)=W(\tau)-\int_0^\tau\beta^{\lambda,y}(s)\, ds, \quad \tau\in [0,t]
$$
is a Wiener process on $[0, t]$ w.r.t. the probability measure $\mathcal{E}^{\lambda,y}(t)d\P$. Equation  \eqref{SDDE-Y-lambda} is just \eqref{SDDE} with $\di W$ changed to $\di W^{\lambda,y}$; thus
\be\label{Girsanov}
\E f(\bY_t^{\lambda,y})\mathcal{E}^{\lambda,y}(t)=\E f(\bX_t^{y}).
\ee
In addition, since $\beta^{\lambda,y}$ is bounded by a constant, we have for each $p\geq 1$, $t>0$
\be\label{E_p}
\sup_{y}\E (\mathcal{E}_t^{\lambda,y})^p<\infty, \quad\sup_{y}\E (\mathcal{E}_t^{\lambda,y})^{-p}<\infty.
\ee

\subsection{Proof of Theorem \ref{t4}}\label{s61} Let $\mathbb{B}$ be a separable Banach space.
We will say that a family of $\mathbb{B}$-valued processes $Z^y(t), t\in [0,T], y\in \CC$ is \emph{$L_p$-Fr\'echet differentiable} at a given point $y\in \CC$, if for any $z\in \CC$ there exists a family $\nabla_z Z^y(t)\in \mathbb{B}, t\in [0, T], z\in \CC$ such that
\be\label{Gato_L_p}
\sup_{t\in [0, T], \|z\|\leq 1}\E\left\|{Z^{y+\eps z}(t)-Z^y(t)\over \eps}-\nabla_z Z^y(t) \right\|^p_{\mathbb{B}}\to 0, \quad \eps\to 0,
\ee
the mappings
\be\label{mapping}
\CC\ni z\mapsto \nabla_z Z^y(t)\in L_p(\Omega, \P, \mathbb{B}),\quad  t\in [0, T]
\ee
are linear, and
\be\label{der_norm_L_p}
\sup_{t\in [0,T]}\sup_{\|z\|\leq 1}\E\|\nabla_z Z^y(t)\|^p_{\mathbb{B}}<\infty.
\ee

We have the following.

\begin{lem}\label{l61} Let $\mathbf{C}^{(1)}$ hold true. Then for any $\lambda\geq 0, T>0, p\geq 1$
           the families
            $$
            \{Y^{\lambda,y}(t), t\in [0, T], y\in \CC\}, \quad \{\bY^{\lambda,y}_t, t\in [0, T], y\in \CC\}
$$
of processes taking values in $\Re^n$ and $\CC$ respectively are $L_p$-Fr\'echet differentiable at any $y\in \CC$. Moreover, the processes
 $U^{\lambda,y,z}(t)=\nabla_z Y^{\lambda,y}(t), t\in [0, T], z\in \CC$,  satisfy
                 \be\label{SDDE-U-1}\ba
\di U^{\lambda,y,z}(t)=\la \nabla a(\bY_t^{\lambda,y}), \bU^{\lambda,y,z}_t \ra\, \di t& +\la \nabla \sigma(\bY_t^{\lambda,y}), \bU^{\lambda,y,z}_t \ra\, \di W(t)
\\&-\lambda\la\nabla\phi(Y^{\lambda,y}(t)-X^x(t)),U^{\lambda,y,z}(t)\ra\, \di t,\quad \bU_0^{\lambda,y,z}=z,
\ea
\ee
and
\be\label{SDDE-U-2}
\nabla_z \bY_t^{\lambda,y}=\bU^{\lambda,y,z}_t , \quad t\in[0,T].
\ee
\end{lem}
\begin{proof} The scheme of the proof actually repeats the one from  the classical proof of $L_2$-differentiability w.r.t. to a parameter of a solution to an SDE; see \cite[Section 2.7]{GSkor72}. Thus we  just briefly outline the usual steps, and focus on particular (minor) difficulties which arise because the state space $\CC$ for the solution to \eqref{SDDE-Y-lambda} is infinite dimensional.

\emph{Step 1.} By assumption $\mathbf{C}^{(1)}$, $a, \sigma$ are (globally) Lipschitz continuous. Thus, applying first It\^o's formula, then the Burkholder-Davis-Gundy inequality, and finally the Gronwall lemma, one gets the bound
\be\label{step1}
\sup_{\|z\|\leq 1}\E \sup_{t\in [0, T]}\left|Y^{\lambda, y+\eps z}(t)-Y^{\lambda, y}(t)\right|^p\leq C\eps^p
\ee
for each fixed $p\geq 1, \lambda\geq 0$. The argument here is  the same as in Section \ref{s41}, thus  we omit the details.

\emph{Step 2.} Denote
$D^{\lambda, y, \eps z}(t)=Y^{\lambda, y+\eps z}(t)-Y^{\lambda, y}(t), t\geq 0$, and let $\bD^{\lambda, y, \eps z}_t, t\geq 0$ be the corresponding segment process. Since $a, \sigma$ are Fr\'echet differentiable and have bounded and uniformly continuous derivatives, it follows from \eqref{step1} that
$$
a(\bY^{\lambda, y+\eps z}_t)-a(\bY^{\lambda, y}_t)=\la \nabla a(\bY^{\lambda, y}_t), \bD^{\lambda, y, \eps z}_t\ra+R^{\lambda, y, \eps z}_a(t),
$$
$$
\sigma(\bY^{\lambda, y+\eps z}_t)-\sigma(\bY^{\lambda, y}_t)=\la \nabla \sigma(\bY^{\lambda, y}_t), \bD^{\lambda, y, \eps z}_t\ra+R^{\lambda, y, \eps z}_\sigma(t)
$$
with
\be\label{step2}
{1\over \eps^p}\sup_{t\in [0, T], \|z\|\leq 1}\Big(\E|R^{\lambda, y, \eps z}_a(t)|^p+\E|R^{\lambda, y, \eps z}_\sigma(t)|^p\Big)\to 0, \quad \eps\to 0.
\ee

\emph{Step 3.}  Using \eqref{step2},  we apply to the pair $\eps^{-1}D^{\lambda, y, \eps z},$ $U^{\lambda, y, z}$ the argument from Step 1, and obtain that
\be\label{step3}
\sup_{\|z\|\leq 1}\E \sup_{t\in [0, T]}\left|\eps^{-1}D^{\lambda, y, \eps z}(t)-U^{\lambda, y, z}(t)\right|^p\to 0, \quad \eps\to 0.
\ee
This is just the relation \eqref{Gato_L_p} for $Z^y(t)=Y^{\lambda, y}(t)$ with $\nabla_y Z^y(t)= U^{\lambda, y, z}(t)$. The linearity of the mapping \eqref{mapping} and the bound \eqref{der_norm_L_p} can be verified straightforwardly, since $U^{\lambda, y, z}$ solves the linear equation \eqref{SDDE-U-1} and $\nabla a, \nabla \sigma, \nabla \phi$, are bounded. This completes the proof of the required differentiability for the $\Re^d$-valued family
$
 \{Y^{\lambda,y}(t)\}$.

\emph{Step 4.} We have already proved \eqref{Gato_L_p} for $Z^y(t)=Y^{\lambda, y}(t)$ with $\nabla_y Z^y(t)= U^{\lambda, y, z}(t)$. These processes are given by \eqref{SDDE-Y-lambda}, \eqref{SDDE-U-1}, thus using  Doob's maximal inequality we can improve this relation and get
$$
\sup_{\|z\|\leq 1}\E\sup_{t\in [0,T]}\left\|{Z^{y+\eps z}(t)-Z^y(t)\over \eps}-\nabla_z Z^y(t) \right\|^p_{\Re^d}\to 0, \quad \eps\to 0.
$$
This yields  \eqref{Gato_L_p} for the $\CC$-valued processes $\wt Z^y(t)=\bY^{\lambda, y}_t$ with $\nabla_z \wt Z^y(t)= \bU^{\lambda, y, z}_t$.
The linearity of \eqref{mapping} for $\nabla_z \wt Z^y(t)$ follows trivially from the same property for $\nabla_z \wt Z^y(t),$ and  the bound \eqref{der_norm_L_p} for $\nabla_z \wt Z^y(t)$ follows from the same bound for $\nabla_z Z^y(t)$ and  Doob's maximal inequality. This completes
the proof of the required differentiability for the $\CC$-valued family
$
 \{\bY^{\lambda,y}_t\}$ and  identity \eqref{SDDE-U-2}.
\end{proof}

\begin{lem}\label{l62} Let $\mathbf{C}^{(1)}$ hold true. Then for any $\lambda\geq 0, T>0, p\geq 1$ the  following hold.
            \begin{enumerate}
 \item The family $\beta^{\lambda,y}(t), t\in [0, T], y\in \CC$ is $L_p$-Fr\'echet differentiable w.r.t. $y$, and
 $$\ba
 \Theta^{\lambda,y,z}(t):=\nabla_z\beta^{\lambda,y}(t)&=\lambda\Big(\nabla\sigma^{-1}\Big)(\bY^{\lambda,y}_t)\bU^{\lambda,y,z}_t\phi(Y^{\lambda,y}(t)-X^x(t))
 \\&+\lambda\sigma(\bY^{\lambda,y}_t)^{-1}\la\nabla\phi(Y^{\lambda,y}(t)-X^x(t)),U^{\lambda,y,z}(t)\ra.
 \ea
 $$
  \item The family $\ell^{\lambda,y}(t)=\log\mathcal{E}^{\lambda,y}(t), t\in [0, T], y\in \CC$ is $L_p$-Fr\'echet differentiable w.r.t. $y$, and
             $$
             \ell^{\lambda,y,z}(t):=\nabla_z \ell^{\lambda,y}(t)=\int_0^t \Theta^{\lambda,y,z}(s)\, \di W_s-\int_0^t\beta^{\lambda,y}(s)\cdot \Theta^{\lambda,y,z}(s)\, \di s.
             $$
             \item The family $\mathcal{E}^{\lambda,y}(t), t\in [0, T], y\in \CC$ is $L_p$-Fr\'echet differentiable w.r.t. $y$, and
             $$
             \nabla_z \mathcal{E}^{\lambda,y}(t)=\mathcal{E}^{\lambda,y}(t)\ell^{\lambda,y,z}(t).
             $$
           \end{enumerate}
\end{lem}

\begin{proof}  Statement \emph{1}  follows from  Lemma \ref{l61} by the chain rule, and straightforwardly implies statement \emph{2}. To prove statement \emph{3}, we first observe that by  statement \emph{2}
\be\label{der_E}
{\mathcal{E}^{\lambda,y+\eps z}(t)-\mathcal{E}^{\lambda,y}(t)\over \eps}\to \mathcal{E}^{\lambda,y}(t)\ell^{\lambda,y,z}(t), \quad \eps\to 0
\ee
in probability uniformly in $t\in [0, T]$.  In addition, by the elementary inequality
$$
|e^a-1|\leq C|a|(1+e^a), \quad a\in \Re,
$$
we have
$$
\left|{\mathcal{E}^{\lambda,y+\eps z}(t)-\mathcal{E}^{\lambda,y}(t)\over \eps}\right|\leq C\left|{\log\mathcal{E}^{\lambda,y+\eps z}(t)-\log\mathcal{E}^{\lambda,y}(t)\over \eps}\right|\Big(\mathcal{E}^{\lambda,y+\eps z}(t)+\mathcal{E}^{\lambda,y}(t)\Big).
$$
By \eqref{Gato_L_p} for the family $\log \mathcal{E}^{\lambda,y}(t), t\in [0, T], y\in \CC$ and \eqref{E_p}, this yields that the left hand side term in \eqref{der_E} has uniformly bounded $L_p$-norm for every $p$, and thus is uniformly $L_{p'}$-integrable for any $p'<p$. This yields that \eqref{der_E} holds true in $L_p$ for any $p$. This proves \eqref{Gato_L_p} for the family $\mathcal{E}^{\lambda,y}(t), t\in [0, T], y\in \CC$; the proofs of linearity for $\eqref{mapping}$ and of the bound \eqref{der_norm_L_p} are  easy and omitted.
\end{proof}

Now it is easy to complete the proof of Theorem \ref{t4}.  By \eqref{Girsanov}, Lemma \ref{l61}, and statements 2,3 of Lemma \ref{l62} applied at the point $y=x$, we get
$$\ba
{1\over \eps}\Big(\E_{x+\eps z}f(\bX_t)&-\E_{x}f(\bX_t)\Big)={1\over \eps}\Big(\E f(\bY^{\lambda,x+\eps z}_t)\mathcal{E}^{\lambda,x+\eps z}(t)-
\E f(\bY^{\lambda,x}_t)\mathcal{E}^{\lambda,x}(t)\Big)
\\&\to \E \la \nabla f(\bY^{\lambda,x}_t), \bU^{\lambda, x,z}_t\ra \mathcal{E}^{\lambda,x}(t)
\\&+\E  f(\bY^{\lambda,x}_t)\mathcal{E}^{\lambda,x}(t)\left(\int_0^t \Theta^{\lambda,x,z}(s)\, \di W_s+\int_0^t\beta^{\lambda,x}(s)\cdot \Theta^{\lambda,x,z}(s)\, \di s\right),\quad \eps\to 0
\ea
$$
uniformly in $t\in [0,T], \|z\|\leq 1$. Note that $Y^{\lambda,x}=X^x$ and thus $\beta^{\lambda,x}(t)\equiv 0$, $\mathcal{E}^{\lambda,x}(t)\equiv 1$. In addition, $\phi(0)=0, \nabla \phi(0)=I_{\Re^n}$ and thus
$$
\Theta^{\lambda,x,z}(t)=\lambda \sigma(\bX^{x}_t)^{-1}U^{\lambda,x,z}(t).
$$
Finally, since $\nabla \phi(0)=I_{\Re^n}$
equation \eqref{SDDE-U-1} for $U^{\lambda,x,z}$ coincides with equation \eqref{U}. That is, $U^{\lambda,x,z}=U^{\lambda,z}$, which completes the proof of Theorem \ref{t4}.

  \subsection{Proof of Theorem \ref{t5}}
First, we fix $h>r$ such that \eqref{lyap} holds true; it is an assumption of Theorem \ref{t23} (and thus of Theorem \ref{t5}) that such $h$ exists. In what follows, the constants may depend on $h$ but we do not indicated this in the notation.

Next, we note that, for any $p\geq 1$ and $Q>0$, one can fix $\lambda$ large enough such that
\be\label{conj}
\E_x\|\bU_t^{\lambda,z}\|^p\leq Ce^{-pQ  t}\|z\|^p, \quad t\geq 0,
\ee
This follows by It\^o's formula and Lemma \ref{lem2} applied to the family of processes $V^{\lambda,v}(t)=|U^{\lambda, z}(t)|^2, \lambda>0, v=|z|^2\in \CC^+_{real}$.
In the sequel we use this inequality for the particular value  $p=(1-\delta)^{-1}$ and a fixed $Q$. Note that by \eqref{conj} we have
\be\label{ff1}
\Big|\E_x \Big\la \nabla f(\bX_t), \bU^{\lambda,z}_t\Big\ra\Big|\leq  C e^{-Q t}\sup_{y\in \CC}\|\nabla f(y)\| \|z\|.
\ee

Next, we note that by \eqref{conj}, $\textbf{H}_3$, and the Burkholder-Davis-Gundy inequality, for any $t_1\leq t_2$
\be\label{conj-int}
\E_x\left|\int_{t_1}^{t_2}\sigma(\bX_s)^{-1}U^{\lambda,z}(s)\,dW(s)\right|^p\leq Ce^{-pQ  t_1}\|z\|^p.
\ee
Since
$$
\E_x \left(f(\bX_t)\int_{t_0}^t\sigma(\bX_s)^{-1}U^{\lambda,z}(s)\,dW(s)\right)=\E_x \left((f(\bX_t)-f(0))\int_{t_0}^t\sigma(\bX_s)^{-1}U^{\lambda,z}(s)\,dW(s)\right)
$$
and
\be\label{H_gamma_tobound}
\sup_x\|f(x)-f(0)\|\leq \|f\|_{H_\gamma},
\ee
 we have
\be\label{ff2}
\left|\lambda \E_x \left(f(\bX_t)\int_{t_0}^t\sigma(\bX_s)^{-1}U^{\lambda,z}(s)\,dW(s)\right)\right|
\leq Ce^{-Qt_0}\|f\|_{H_\gamma}\|z\|, \quad t_0\leq t.
\ee
 On the other hand, by the Markov property,
$$
\E_x \left(f(\bX_t)\int_0^{t_0}\sigma(\bX_s)^{-1}U^{\lambda,z}(s)\,dW(s)\right)=\E_x \left(P_{t-t_0}f(\bX_{t_0})\int_0^{t_0}\sigma(\bX_s)^{-1}U^{\lambda,z}(s)\,dW(s)\right).
$$
 Denote
$$
\ov f=\int_\CC f(y)\, \pi (\di y).
$$
By \eqref{KR_thm} and  \eqref{rate_ind},
$$
|P_sf(x)-\ov f|\leq C{1\over r(cs)^\delta} \phi(V(x))^\delta\|f\|_{H_\gamma}.
$$
Clearly,
$$
\E_x \left(\ov f \int_0^{t_0}\sigma(\bX_s)^{-1}U^{\lambda,z}(s)\,dW(s)\right)=0,
$$
hence by \eqref{conj-int} with $t_1=0, t_2=t_0$ and H\"older's inequality applied to $p=(1-\delta)^{-1}, q=\delta^{-1}$, we have
\be\label{ff3}
\left|\E_x \left(f(\bX_t)\int_0^{t_0}\sigma(\bX_s)^{-1}U^{\lambda,z}(s)\,dW(s)\right)\right|\leq
 C{1\over r(c(t-t_0))^\delta} \Big(\E_x \phi(V(\bX_{t_0}))\Big)^\delta\|f\|_{H_\gamma}\|z\|.
\ee
By Jensen's inequality, $\E_x \phi(V(\bX_{t_0}))\leq \phi(\E_x V(\bX_{t_0}))$. On the other hand, if $t_0=kh, k\in \mathbb{N}$, then it follows from  \eqref{lyap} that
$$
\E_x V(\bX_{t_0})\leq V(x)+k C_V.
$$
Since $\phi$ is concave and non-negative, we have
\be\label{subad}
\phi(u+v)\leq \phi(u)+\phi(v)-\phi(0)\leq \phi(u)+\phi(v).
\ee
Hence
\be\label{ff4}
\Big(\E_x \phi(V(\bX_{t_0}))\Big)^\delta\leq C(t_0+\phi(V(x)))^\delta, \quad t_0\in h\mathbb{N}.
\ee

Now we can finish the proof. Since the function $\phi(v)$ is sub-linear (see \eqref{subad}), the function $r(t)$ is  sub-exponential. In particular,
we can fix $c_Q\in (0, c/2)$ small enough such that
$$
 \log r(c_Q t)\leq {Qt\over 2}, \quad t\geq 0.
$$
Take
$$
t_0=h\Big[h^{-1}Q^{-1}\log r(c_Qt)\Big]\in h\mathbb{N},
$$
then $t-t_0\geq {t/2}$ and
$$
r(c(t-t_0))\geq r(c_Q t)
$$
because $r(\cdot)$  is increasing, and $c_Q<c/2$. Combining the representation \eqref{der_rep} and \eqref{ff1} -- \eqref{ff4}, we get
$$
\ba
|\nabla_z \E_x f(\bX_t)|&\leq Ce^{-Qt}\sup_{y\in \CC}\|\nabla f(y)\| \|z\|
\\&+{C\over r(c_Q t)}\|f\|_{H_\gamma} \|z\|+{C\over r(c_Q t)}(Q^{-1}\log r(c_Qt)
+\phi(V(x)))^\delta\|f\|_{H_\gamma} \|z\|,\quad x,y\in \CC,
\ea
$$
which after simple re-arrangement gives \eqref{der_rate_ind}.

  \subsection{Proof of Theorem \ref{t6}.}
  For $k>1$, the argument remains principally the same as the one developed for $k=1$ in the  two previous sections, with just technical complications which
  makes the proof more cumbersome. Thus we just outline the main steps of the proof, paying particular attention to one new circumstance; see \emph{Step 5} below.  Everywhere below we assume $\mathbf{C}^{(k)}$ to hold for some $k>1$.

 \emph{ Step 1.} By iteration of the argument in the proof of Lemma \ref{l61}, we obtain the following: the family  $Y^{\lambda,y}(t), t\in [0, T], y\in \CC$ is $k$ times $L_p$-Fr\'echet differentiable w.r.t. $y$, and the corresponding direction-wise derivatives $U^{\lambda,y,z_1, \dots,z_k}(t)=\nabla_{z_1}\dots\nabla_{z_k} Y^{\lambda,y}(t)$ satisfy SDDEs of the form
                 \be\label{SDDE-U-k}\ba
\di U^{\lambda,y,z_1, \dots, z_k}(t)&=\la \nabla a(\bY_t^{\lambda,y}), \bU^{\lambda,y,z_1, \dots, z_k}_t \ra\, \di t +\la \nabla \sigma(\bY_t^{\lambda,y}), \bU^{\lambda,y,z_1, \dots, z_k}_t \ra\, \di W(t)
\\&-\lambda\la\nabla\phi(Y^{\lambda,y}(t)-X^x(t)),U^{\lambda,y,z_1, \dots, z_k}(t)\ra\, \di t
\\&+D^{\lambda,y,z_1, \dots, z_k}(t)\, dt+S^{\lambda,y,z_1, \dots, z_k}(t)\,dW(t),\quad \bU_0^{\lambda,y,z_1, \dots, z_k}=0.
\ea
\ee
The terms  $D^{\lambda,y,z_1, \dots, z_k}(t), S^{\lambda,y,z_1, \dots, z_k}(t)$ can be represented as sums of various $l$-linear forms, which are expressed in terms of $\nabla^i a(\bY_t^{\lambda,y}),  \nabla^i \sigma(\bY_t^{\lambda,y}), \nabla^i\phi(Y^{\lambda,y}(t)-X^x(t)), i=1, \dots, k$ (and thus are bounded). The arguments in each of those  $l$-linear forms have the generic form
$$
\bU^{\lambda,y,z_{I_1}}_t, \dots, \bU^{\lambda,y,z_{I_l}}_t,
$$
where $I_1, \dots, I_l$ is a disjoint partition of $\{1, \dots, k\}$, and we use the notation  $I=\{i_1, \dots, i_j\}$
$$
\bU^{\lambda,y,z_{I_1}}_t=\bU^{\lambda,y,z_{i_1}, \dots, z_{i_j}}_t, \quad I=\{i_1, \dots, i_j\}.
$$

 \emph{ Step 2.} We have for $\lambda$ large enough  \be\label{conj_k}
\E\|\bU_t^{\lambda,x,z_1, \dots, z_k}\|^p\leq Ce^{-pQ  t}\|z_1\|^p\dots\|z_k\|^p, \quad t\geq 0.
\ee
To see this, first recall that  $Y^{\lambda, x}=X^x$, $\nabla\phi(0)=I_{\Re^n}.$ Then \eqref{SDDE-U-k} for $y=x$ transforms to
                \be\label{SDDE-U-k-x}\ba
\di U^{\lambda,x,z_1, \dots, z_k}(t)&=\la \nabla a(\bX_t^{x}), \bU^{\lambda,x,z_1, \dots, z_k}_t \ra\, \di t +\la \nabla \sigma(\bX_t^{x}), \bU^{\lambda,x,z_1, \dots, z_k}_t \ra\, \di W(t)
\\&-\lambda U^{\lambda,x,z_1, \dots, z_k}(t)\, \di t
+D^{\lambda,x,z_1, \dots, z_k}(t)\, dt+S^{\lambda,x,z_1, \dots, z_k}(t)\,dW(t),\quad \bU_0^{\lambda,y,z_1, \dots, z_k}=0,
\ea
\ee
Now {the proof} can be made inductively: assuming \eqref{conj_k} holds true for $1, \dots, k-1$, we have
$$
\E|D^{\lambda,x,z_1, \dots, z_k}(t)|^p+\E|S^{\lambda,x,z_1, \dots, z_k}(t)|^p\leq Ce^{-pQ  t}\|z_1\|^p\dots\|z_k\|^p, \quad t\geq 0,
$$
which yields \eqref{conj_k} for $k$ by Lemma \ref{lem3}.

\emph{Step 3.} Similarly to the proof of Lemma \ref{l62} we get that the following families are $k$ times $L_p$-Fr\'echet differentiable w.r.t. $y$:
\begin{itemize}
  \item $\beta^{\lambda,y}(t), t\in [0, T], y\in \CC$;
  \item $\ell^{\lambda,y}(t):=\log \mathcal{E}^{\lambda,y}(t), t\in [0, T], y\in \CC$;
  \item $\mathcal{E}^{\lambda,y}(t_1;t_2)=\mathcal{E}^{\lambda,y}(t_1)^{-1}\mathcal{E}^{\lambda,y}(t_2), t_1,t_2\in [0, T], y\in \CC$.
\end{itemize}
The corresponding direction-wise derivatives will be defined by
$$
 \Theta^{\lambda,y,z_1, \dots, z_k}(t)=\nabla_{z_1}\dots\nabla_{z_k}\beta^{\lambda,y}(t), \quad \ell^{\lambda,y,z_1, \dots, z_k}(t)=\nabla_{z_1}\dots\nabla_{z_k}\ell^{\lambda,y}(t),
 $$
 $$
 \mathcal{E}^{\lambda,y,z_1, \dots, z_k}(t_1;t_2)=\nabla_{z_1}\dots\nabla_{z_k}\mathcal{E}^{y,\lambda}(t_1;t_2).
 $$
 We have straightforwardly
 $$\ba
             \ell^{\lambda,y,z_1, \dots, z_k}(t)=\int_0^t \Theta^{\lambda,y,z_1, \dots, z_k}(s)\, \di W_s
            +\sum_{i=1, \dots, k}\int_0^t\Theta^{\lambda,y, z_j}(s)\cdot \Theta^{\lambda,y,z_1, \dots,z_{j-1}, z_{j+1}, \dots z_k}(s)\, \di s.
             \ea
             $$
 \emph{Step 4.} The derivative $\Theta^{\lambda,x,z_1, \dots, z_k}(t)$ can be expressed as a sum of bounded $l$-linear forms $(l=1, \dots, k)$ applied to
$$
\bU^{\lambda,x,z_1', \dots, z_j'}_t, \quad j=1, \dots, k, \quad z_1', \dots, z_j'\in\{z_1, \dots, z_k\}.
$$
Then by \eqref{conj_k} we have that for  $\lambda$ large enough
 \be\label{conj_k_Theta}
\E|\Theta^{\lambda,x,z_1, \dots, z_k}(t)|^p\leq Ce^{-pQ  t}\|z_1\|^p\dots\|z_k\|^p, \quad t\geq 0.
\ee
This implies
\be\label{conj_k_ell_2}
\E|\ell^{\lambda,x,z_1, \dots, z_k}(t_2)-\ell^{\lambda,x,z_1, \dots, z_k}(t_1)|^p\leq C e^{-pQ  t_1} \|z_1\|^p\dots\|z_k\|^p, \quad t_2\geq t_1\geq 0.
\ee
Recall that $\mathcal{E}^{\lambda,x}(t_1;t_2)=1,$ hence the derivative $\mathcal{E}^{\lambda,y,z_1, \dots, z_k}(t_1;t_2)$ is a polynomial of
$$
\ell^{\lambda,x,z_1', \dots, z_j'}(t_2)-\ell^{\lambda,x,z_1', \dots, z_j'}(t_1), \quad j=1, \dots, k, \quad z_1', \dots, z_j'\in\{z_1, \dots, z_k\}.
$$
Therefore for  $\lambda$ large enough
\be\label{conj_k_E}
\E|\mathcal{E}^{\lambda,x,z_1, \dots, z_k}(t_1;t_2)|^p\leq C e^{-pQ  t_1} \|z_1\|^p\dots\|z_k\|^p,  \quad t_2\geq t_1\geq 0.
\ee

\emph{Step 5.} By \eqref{Girsanov}, we have
\be\label{int_rep_der_k}
\ba
\nabla_{z_1}&\dots\nabla_{z_k}\E_x f(\bX_t)=\E \nabla_{z_1}\dots\nabla_{z_k}\Big(f(\bY^{\lambda,y}_t)\mathcal{E}^{\lambda,y}(t)\Big)|_{y=x}
\\&=\sum_{I\bigcup J=\{1, \dots,k\}, I\bigcap J=\varnothing}\E \nabla^{I}_{z_I}\Big(f(\bY^{\lambda,y}_t)\Big)|_{y=x}\mathcal{E}^{\lambda,x, z_J}(t)
\\&=\sum_{I\bigcup J=\{1, \dots,k\}, I\bigcap J=\varnothing, I\not=\varnothing}\E \nabla^{I}_{z_I}\Big(f(\bY^{\lambda,y}_t)\Big)|_{y=x}\mathcal{E}^{\lambda,x, z_J}(t)+\E f(\bX^{x}_t)\mathcal{E}^{\lambda,x, z_1, \dots, z_k}(t)
\ea
\ee
where the following notation is used:
\begin{itemize}
  \item for $I=\{i_1, \dots, i_m\}$,
  $$
  \nabla^{I}_{z_I}=\nabla_{z_{i_1}}\dots\nabla_{z_{i_m}};
  $$
  \item for $J=\{j_1, \dots, j_r\}$,
  $$
  \mathcal{E}^{\lambda,x, z_J}(t)=\mathcal{E}^{\lambda,x, z_J}(0,t), \quad \mathcal{E}^{\lambda,x, z_J}(t_1,t_2)=\mathcal{E}^{\lambda,x, z_{j_1}, \dots, z_{j_r}}(t_1,t_2).
  $$
  \end{itemize}
By \eqref{conj_k} and \eqref{conj_k_E} with $t_2=t, t_1=0$, we have
\be\label{est1}
\left|\sum_{I\bigcup J=\{1, \dots,k\}, I\bigcap J=\varnothing, I\not=\varnothing}\E \nabla^{I}_{z_I}\Big(f(\bY^{\lambda,y}_t)\Big)|_{y=x}\mathcal{E}^{\lambda,x, z_J}(t)\right|\leq C\|f\|_{(k)}e^{-Q  t}\|z_1\|\dots\|z_k\|.
\ee
That is, the first term (the sum) on the right hand side of \eqref{int_rep_der_k} admits  an estimate completely analogous to \eqref{ff1}.

The second term on the right hand side of \eqref{int_rep_der_k} is analogous to the second term in \eqref{der_rep}. A slight new difficulty which appears in the case $k>1$ is that now this term can not be simply written by means of a stochastic integral: $\mathcal{E}^{\lambda,x, z_1, \dots, z_k}(t)$ is actually a mixture of various multiple stochastic integrals.  This difficulty can be resolved by the following trick, which  makes it possible to estimate this term without a study of its inner structure. We have for arbitrary $t_0\leq t$
\be\label{E_dec}\ba
\mathcal{E}^{\lambda,x, z_1, \dots, z_k}(t)&=\nabla_{z_1}\dots\nabla_{z_k}\Big(\mathcal{E}^{\lambda,y}(0; t_0)\mathcal{E}^{\lambda,y}( t_0;t)\Big)|_{y=x}=\sum_{I\bigcup J=\{1, \dots,k\}, I\bigcap J=\varnothing}\mathcal{E}^{\lambda,x, z_I}(0;t_0)\mathcal{E}^{\lambda,x, z_J}(t_0;t).
\ea
\ee
By \eqref{H_gamma_tobound} and  \eqref{conj_k_E},
\be\label{est2}
\left| \E f(\bX^x_t)\sum_{I\bigcup J=\{1, \dots,k\}, I\bigcap J=\varnothing, J\not=\varnothing} \mathcal{E}^{\lambda,x, z_I}(0;t_0)\mathcal{E}^{\lambda,x, z_J}(t_0;t)\right|\leq  C\|f\|_{H_\gamma}e^{-Q  t_0}\|z_1\|\dots\|z_k\|,
\ee
which is a straightforward analogue to \eqref{ff2}. The term with $J=\varnothing$ equals
$$
\E f(\bX^{x}_t)  \mathcal{E}^{\lambda,x,z_1, \dots, z_k}(t_0)=\E P_{t-t_0}f(\bX_{t_0})\mathcal{E}^{\lambda,x,z_1, \dots, z_k}(t_0).
$$
Note that
$$
\E\mathcal{E}^{\lambda,x,z_1, \dots, z_k}(t_0)=\nabla_{z_1}\dots\nabla_{z_k}\E\mathcal{E}^{\lambda,y}(t_0)|_{y=x}=\nabla_{z_1}\dots\nabla_{z_k}1=0.
$$
Repeating literally the calculations used in the proof of \eqref{ff3} and using \eqref{ff4}, we get
\be\label{est3}
\left|\E f(\bX_t^x) \mathcal{E}^{\lambda,x,z_1, \dots, z_k}(t_0)\right|\leq
 C{1\over r(c(t-t_0))^\delta} \Big(t_0+\phi(V(x))\Big)^\delta\|f\|_{H_\gamma}\|z\|.
\ee
Using \eqref{est1}, \eqref{est2}, \eqref{est3}, and repeating  the optimization in $t_0$ procedure from the last part of the proof of Theorem \ref{t5}, we complete the proof of the theorem.
\appendix
\section{The Kullback-Leibler distance and related bounds}
For a pair of probability measures $\mu\ll \nu$ on a measurable space $(X,\mathcal{X})$  \emph{the  Kullback--Leibler $($KL--$)$ divergence of $\mu$ from $\nu$}  is defined by
$$
D_{KL}(\mu\|\nu):=\int_X \log {\frac{d \mu}{d \nu}}\, d \mu=\int_X \frac{d \mu}{ d \nu}\log\Bigl(\frac{d \mu}{d \nu}\Bigr)\, d \nu.
$$
 The  $KL$--divergence is known to be a stronger measure of difference between probability distributions than the total variation distance; in particular, the following \emph{Pinsker inequality} hold true, e.g. \cite[Lemma~2.5.(i)]{Ts}:
 \be\label{otzenka2}
d_{TV}(\mu,\nu)\le \sqrt{{\frac12} D_{KL}(\mu\|\nu)}.
\ee
In addition, the $KL$--divergence yields the following lower bound, see \cite[Lemma A.1]{BKS18}: for any $N>1$ and any set $A\in\mathcal{X}$,
\begin{equation}
\label{diffbound}
\nu(A)\ge\frac1N \mu(A)-\frac{D_{KL}(\mu\|\nu)+\log 2}{N\log N}.
\end{equation}
Next, let  $\xi$ be an  $m$-dimensional   It\^o process with $\xi_0=0$ and
\begin{equation}\label{Ito_process}
\di \xi_t=\beta_td t+\di\, W_t, \quad t\ge0,
\end{equation}
where $W$ is a Wiener process in $\Re^m$, and $(\beta_t)_{t\ge 0}$ is  a progressively measurable  process.
The following bound is available for the $KL$--divergence  of the law $\mu_\xi$  of the process $\xi$ on $C([0,\infty), \Re^m)$ w.r.t. of the law $\mu_W$ of $W$ (the Wiener measure).

\begin{thm}\label{tKLGirsanov}\cite[Theorem A.2]{BKS18}
 $$
 D_{KL}(\mu_\xi\|\mu_W)\leq {\frac12}\E\int_0^\infty |\beta_t|^2\, d t.
 $$
 \end{thm}

\section{Auxiliary tail- and $L_p$-estimates}\label{sB}

%In this section we collect auxiliary tail/moment estimates for controlled systems used in the main proofs.
The following lemma was suggested by R.~Schilling.

\begin{lem}\label{lem1} Let $V(t)\geq 0$ be an It\^o process with
$$
\di V(t)=\eta(t)\, \di t+\, \di M(t),
$$
where $M$ is a continuous local martingale with quadratic variation
$$
\langle M\rangle(t)=\int_0^t m(s)\, \di s, \quad t\geq 0.
$$
Let for some constants $A\ge 0, B >0,\lambda >0$ and a random variable $\tau\geq 0$
$$
\eta(t)\leq -\lambda U(t)+A, \quad m(t)\leq B, \quad t\leq \tau.
$$
Assume also that $\tau\leq T$ for some constant $T>0$.

Then for every $\delta \in (0,1/2)$ there exist $C_{1}$, $C_{2}>0,$ which depend only on $\delta$ and $T$,  such that
$$
\P\left(\sup_{t\leq \tau}\Big(V(t)- e^{-\lambda t}V(0)\Big)\geq A\lambda^{-1}+B^{1/2} \lambda^{-\delta}R\right)\leq C_1 e^{-C_2  R^2}, \quad R\geq 0.
$$
\end{lem}
\begin{proof} We have \be\label{decomp}
V(t)=e^{-\lambda t}V(0)+\int_0^te^{-\lambda(t-s)}\xi(s)\, \di s+\int_0^te^{-\lambda(t-s)}\, \di M(s),
\ee
where
$$
\xi(t)=\eta(t)+\lambda V(t)\leq A, \quad t\leq \tau.
$$
Clearly,
$$
\int_0^te^{-\lambda(t-s)}\xi(s)\, \di s \leq A\lambda^{-1}(1-e^{-\lambda t}), \quad t\leq \tau,
$$
and we have to study the third term in \eqref{decomp}, only. Without loss of generality, we can assume that  $M(0)=0$. Extending the probability space, if necessary,
we find a standard Brownian motion $\wt W$ such that
$$
\int_0^te^{\lambda s}\, \di M(s) = B^{1/2} \wt W \Big(\int_0^t e^{2 \lambda s}\frac{m(s)}{B}\,\di s\Big),\;t \ge 0.
$$
Therefore,
\begin{align*}
\sup_{0\le t \le \tau} \int_0^te^{-\lambda(t-s)}\, \di M(s)&= B^{1/2} \sup_{0\le t \le \tau} \Big\{e^{-\lambda t}  \wt W \Big(\int_0^t e^{2 \lambda s}\frac{m(s)}{B}\,\di s\Big)\Big\}\\
                                                           &=B^{1/2}\sup_{0\le t \le \tau} \Big\{e^{-\lambda t}  \sup_{0\le u \le t} \wt W \Big(\int_0^u e^{2 \lambda s}\frac{m(s)}{B}\,\di s\Big)\Big\}\\
                                                           &\le B^{1/2} \sup_{0\le t \le \tau} \Big\{e^{-\lambda t}  \sup_{0\le u \le t} \wt W \Big(\int_0^u e^{2 \lambda s}\,\di s\Big)\Big\},\\
\end{align*}
where the second equality holds true since $t \mapsto e^{-\lambda t}$ is decreasing. Next, there exists a standard Brownian motion $\wh W$ such that
$$
\wt W \Big(\int_0^u e^{2 \lambda s}\,\di s\Big)=\int_0^u e^{\lambda s}\,\di \wh W(s), \quad u\geq 0.
$$
Using once again that $t \mapsto e^{-\lambda t}$ is decreasing we get finally
$$
\sup_{0\le t \le \tau} \int_0^te^{-\lambda(t-s)}\, \di M(s)\leq B^{1/2} \sup_{0\le t \le \tau} \int_0^te^{-\lambda(t-s)}\, \di \wh W(s).
$$
For $N>0$ denote
$$
D_N^{\delta, T}=\Big\{ |\wh W(t)-\wh W(s)|\leq N |t-s|^\delta;\,s,t\leq T\Big\},
$$
and observe that, by Fernique's theorem,
\be\label{M_hoelder}
\P\left(D_N^{\delta, T}\right)\leq c_1 e^{-c_2 N^2}, \quad N>0
\ee
with constants $c_1, c_2$ which depend only on $\delta$ and $T$.
%To prove this bound, we recall that,
%extending the probability space, if necessary,  one  can represent
%$$
% M(t)=\wt W\Big(\langle M\rangle (t)\Big),
%$$
%where $\wt W$ is a scalar Wiener process.  We have
%$$
%\langle M\rangle (t)\leq B\tau\leq B_0T, \quad t\leq \tau.
%$$
%On the other hand,
%$$
%\P\left(\sup_{s\not= t, s,t\in [0,B_0T]}{|\wt W(t)-\wt W(s)|\over |t-s|^\delta}>N\right)\leq c_1 e^{-c_2 N^2},
%$$
%which immediately gives \eqref{M_hoelder}.
%
%We have
%$$
%\langle M\rangle(t)-\langle M\rangle(s)\leq B(t-s), \quad s\leq t\leq \tau,
%$$
Observe that
$$
\int_0^te^{-\lambda(t-s)}\, \di \wh W(s) = \int_0^t\big(\wh W(t)-\wh W(s)\big)\lambda e^{-\lambda(t-s)}\, \di s+e^{-\lambda t} \wh W(t).
$$
Therefore, on the set set $D_N^{\delta, T}$, we obtain for $t\leq \tau\leq T$
$$
\ba
\int_0^te^{-\lambda(t-s)}\, \di M(s)&\leq N B^{1/2}\int_0^t(t-s)^\delta\lambda e^{-\lambda(t-s)}\, \di s
+NB^{1/2} e^{-\lambda t}t^\delta
\\&\leq N B^{1/2}\lambda^{-\delta}\left(\Gamma(\delta)+\sup_{x>0} x^\delta e^{-x}\right).
\ea
$$
Taking
$$
N=\left(\Gamma(\delta)+\sup_{x>0} x^\delta e^{-x}\right)R
$$
and using \eqref{M_hoelder}, the proof is complete.

\end{proof}

Denote $\CC^+_{real}$ the set of non-negative functions in $C([-r,0], \mathbb{R})$.
\begin{lem}\label{lem2}
  For each $\lambda>0, v\in \CC^+_{real}$, let $V^{\lambda,v}(t)\geq 0$, $t \ge -r$ be an adapted process with continuous paths such that
  $$
\di V^{\lambda,v}(t)=\eta^{\lambda,v}(t)\, \di t+\, \di M^{\lambda,v}(t),\quad t\geq 0, \quad  \bV^{\lambda,v}_0=v\in \CC^+_{real},
$$
where  $M^{\lambda,v}$ is a continuous local martingale and  for some constant $K \ge 0$
$$
\eta^{\lambda,v}(t)\leq K\|\bV^{\lambda,v}_t\|-\lambda V^{\lambda,v}(t), \quad {\di \langle M^{\lambda,v}\rangle(t)\over \di t}\leq K^2\|\bV^{\lambda,v}_t\|^2.
$$
Then for each $p \ge 1$ and $Q>0$ there exist $\lambda_{p,Q, K}>0$ and a constant $C_{p,Q,K}$ such that for $\lambda\geq \lambda_{p,Q, K}$
$$
\E \|\bV^{\lambda,v}_t\|^p \le C_{p,Q,K} e^{-Qt}\|v\|^p,\quad t \ge 0, \quad v\in \CC^+_{real}.
$$
%$$
%\P\big(U^\lambda (t)\exp\{qt\}\ge R \mbox{ for some } t \ge 0\big)\le c_1\exp\{-c_2R^2\},\,R \ge 0,
%$$
\end{lem}
\begin{proof} First, we note that by the Gronwall lemma and the Burkholder-Davis-Gundy inequality for each $T>0, p\geq 1$ there exists a constant $C_{p,T,K}$ such that
\be\label{BDG}
\E \sup_{t\leq T} V^{\lambda,v}(t)^p\leq C_{p,T,K}\|v\|^p;
\ee
the argument here is the same as in Section \ref{s41}. That is, to prove the required statement it is enough to find for given $p, Q$ some $T=T_{p,Q,K}>0$ and $\lambda_{p,Q, K}>0$ such that
\be\label{est4}
\E \|\bV^{\lambda,v}_t\|^p \le C_{p,Q,K} e^{-QT}\|v\|^p, \quad \lambda\geq \lambda_{p,Q, K}, \quad v\in \CC^+_{real}.
\ee
We fix $T=2r$ and put
$$
\tau^{\lambda, v}=\inf\{t: |V^{\lambda,v}(t)|\geq 2\|v\|\}\wedge T.
$$
Then the assumption of Lemma 4.1 holds true for $V=V^{\lambda, v}, \tau=\tau^{\lambda, v}$ with $A=2K\|v\|$, $B=4K^2\|v\|^2$. Note that this lemma yields
$$
V^{\lambda,v}(t)\leq e^{-\lambda t}v(0)+{A\over \lambda}+{B^{1/2}\over \lambda^{\delta}}\Xi^{\lambda,v}(t), \quad t\leq\tau^{\lambda,v}
$$
with a fixed $\delta<1/2$ and
$$
\|\sup_{t\leq T}\Xi^{\lambda,v}(t)\|_{L_p}\leq C_{p,T}'.
$$
We will take $\lambda_{p,Q,K}\geq 4K$, then for $\lambda\geq \lambda_{p,Q,K}$
$$
V^{\lambda,v}(t)1_{t\leq \tau^{\lambda,v}}\leq
\left(e^{-\lambda t}+{2K\over \lambda}+{2K\over \lambda^\delta}\Xi^{\lambda,v}(t)\right)\|v\|\leq \left({3\over 2}+{2K\over \lambda^\delta}\Xi^{\lambda,v}(t)\right)\|v\|,
$$
which yields
$$
\P(\tau^{\lambda,v}<T)\leq \P\left(\Xi^{\lambda,v}(\tau^{\lambda,v})\geq {\lambda^\delta\over 4K}\right)\leq C_{p,T}'(4K)^p \lambda^{-\delta p}.
$$
Then by \eqref{BDG} and the Cauchy inequality
$$
\ba
\E \sup_{t\in[0, T]}V^{\lambda,v}(t)^p1_{\tau^{\lambda,v}<T}&\leq \left(\E\sup_{t\leq T} V^{\lambda,v}(t)^{2p}\right)^{1/2}\left(\P(\tau^{\lambda,v}<T)\right)^{1/2}
\leq  {(C_{2p,T}C_{p,T}'(4K)^p)^{1/2} \over \lambda^{p\delta/2}}\|v\|^p.
\ea
$$
Combining these calculations, we get
$$
V^{\lambda,v}(t)\leq \left(e^{-\lambda t}+{2K\over \lambda}+\left({2K\over \lambda^{\delta/2}}\vee \sqrt{2K\over \lambda^{\delta/2}}\right)\Delta^{\lambda,v,K}\right)\|v\|, \quad t\in [0, T]
$$
with
$$
\|\Delta^{\lambda,v,K}\|_{L_p}\leq C_{p,T}'.
$$
In particular
$$
\|\bV^{\lambda,v}_T\|\leq \left(e^{-\lambda T/2}+{2K\over \lambda}+{2K\over \lambda^{\delta/2}}\Delta^{\lambda,v,K}\right)\|v\|,
$$
which  easily yields \eqref{est4} for $\lambda\geq \lambda_{p,Q,K}$ and $\lambda_{p,Q,K}$ large enough.
\end{proof}

The following lemma can be proved by essentially the same argument; we leave the details for the reader.

\begin{lem}\label{lem3} Let $k>1$, and for each $\lambda>0, v_1, \dots, v_k\in \CC^+_{real}$ let $V^{\lambda,v_1, \dots, v_k}(t)\geq 0$, $t \ge -r$ be an adapted process with continuous paths such that
  $$
  \di V^{\lambda,v_1, \dots, v_k}(t)=\eta^{\lambda,v_1, \dots, v_k}(t)\, \di t+\, \di M^{\lambda,v_1, \dots, v_k}(t),\quad t\geq 0, \quad  \bV^{\lambda,v_1, \dots, v_k}_0=0,
  $$
  where $M^{\lambda,v_1, \dots, v_k}$ is a continuous local martingale and, for some constant  $K \ge 0$,
$$
\eta^{\lambda,v_1, \dots, v_k}(t)\leq K\|\bV^{\lambda,v_1, \dots, v_k}_t\|-\lambda V^{\lambda,v_1, \dots, v_k}(t)+L^{\lambda,v_1, \dots, v_k}(t),
$$
$$ {\di \langle M^{\lambda,v_1, \dots, v_k}\rangle(t)\over \di t}\leq K^2\|\bV^{\lambda,v_1, \dots, v_k}_t\|^2+N^{\lambda,v_1, \dots, v_k}(t).
$$
Fix $p \ge 1$ and assume that the non-negative adapted processes $L^{\lambda,v_1, \dots, v_k}(t)$, $N^{\lambda,v_1, \dots, v_k}(t), t\geq 0$ are such that
for each $Q>0$ there exist $\lambda_{p,Q}^0>0$ and a constant $C_{p,Q}^0$ such that
$$
\E \|\mathbf{L}^{\lambda,v_1, \dots, v_k}_t\|^p \le C_{p,Q}^0 e^{-Qt}\|v_1\|^p\dots\|v_k\|^p, \quad \E \|\mathbf{N}^{\lambda,v_1, \dots, v_k}_t\|^{2p} \le C_{p,Q}^0 e^{-Qt}\|v_1\|^{p}\dots\|v_k\|^{p},\quad t \ge 0
$$
for any $\lambda\geq \lambda_{p,Q}^0$ and $v_1, \dots, v_k\in \CC^+_{real}$.

Then for each $Q>0$ there exist $\lambda_{p,Q,K}>0$ and a constant $C_{p,Q,K}$ such that for $\lambda\geq \lambda_{p,Q,K}$
$$
\E \|\bV^{\lambda,v_1, \dots, v_k}_t\|^p \le C_{p,Q,K} e^{-Qt}\|v_1\|^p\dots\|v_k\|^p,\quad t \ge 0, \quad v_1, \dots, v_k\in \CC^+_{real}.
$$
%The constants $\lambda_{p,Q},C_{p,Q}$ depend only on $p, Q, K$ and $\lambda_{p',Q'}, C_{p',Q'}^0, p'\geq 1, Q'>0$.
%$$
%\P\big(U^\lambda (t)\exp\{qt\}\ge R \mbox{ for some } t \ge 0\big)\le c_1\exp\{-c_2R^2\},\,R \ge 0,
%$$
\end{lem}

\end{document}